\documentclass[11pt,
]{amsart}
\usepackage{
amssymb, graphicx
}
 \usepackage[foot]{amsaddr}
\usepackage{amsthm}
\usepackage{xcolor}
\usepackage{graphicx,color}
\usepackage{mathtools}
\usepackage{bbm}
\usepackage{amsmath}
\newtheorem{theorem}{Theorem}
\newtheorem{lemma}{Lemma}

\newtheorem{remark}{Remark}

\theoremstyle{remark}

\newcommand{\be}[1]{\begin{equation}\label{#1}} 
\newcommand{\ee}{\end{equation}}

\title[Transverse and mixed ray transforms]{Invertibility of local geodesic transverse and mixed ray transforms I: basic cases}


    \author[G. Uhlmann]{Gunther Uhlmann}
\address{G. Uhlmann: Department of Mathematics, University of Washington, Seattle, WA 98195, USA; Institute for Advanced Study, 
The Hong Kong University of Science and Technology, Kowloon, Hong Kong, China (\tt{gunther@math.washington.edu})
}
  
  \author[J. Zhai]{Jian Zhai}
\address{J. Zhai: School of Mathematical Sciences, Fudan University, 220 Handan Road, Shanghai 200433, China
  (\tt{jianzhai@fudan.edu.cn}).}
\thanks{J. Zhai is supported by National Key Research and Development Programs of China (No. 2023YFA1009103), Science and Technology Commission of Shanghai Municipality (23JC1400501)}
 
\begin{document}

\begin{abstract}
Consider a compact Riemannian manifold in dimension $n\geq 3$ with strictly convex boundary. We show that the transverse ray transform of $1$ tensors and the mixed ray transform of $1+1$ tensors are invertible, up to natural obstructions, near a boundary point. When the manifold admits a strictly convex function, this local invertibility result leads to a global result by a layer stripping argument.
\end{abstract}

\maketitle 
\section{Introduction}
Let $M$ be an $n$ dimensional manifold. We use the notation $T_pM$ for the tangent space of $M$ at $p\in M$, $T_p^*M$ for the cotangent space. Denote 
\[
T^k_\ell TM=\coprod_{p\in M}\underbrace{T_pM\otimes\cdots\otimes T_pM}_{k}\otimes\underbrace{T_p^*M\otimes\cdots\otimes T_p^*M}_{\ell},
\]
for the bundle of $k$-contravariant, $\ell$-covariant tensors, where $\coprod$ denotes the disjoint union. Under local coordinates $(x^1,x^2,\cdots, x^n)$, each smooth section $f$ of $T^k_\ell TM$ can be expressed as
\[
f(p)=f^{i_1\cdots i_k}_{j_1\cdots j_\ell}(p)\partial_{x^{i_1}}\otimes\cdots\otimes\partial_{x^{i_k}}\otimes\mathrm{d}x^{j_1}\otimes\cdots\otimes\mathrm{d}x^{j_\ell},
\]
where $f^{i_1\cdots i_k}_{j_1\cdots j_\ell}$ is a smooth function on $M$. Here and throughout the paper, we use the Einstein's summation convention. If $f^{i_1\cdots i_k}_{j_1\cdots j_\ell}$ is symmetric in the indices $(i_1\cdots i_k)$ and also in the indices $(j_1\cdots j_\ell)$, then $f$ is called a symmetric $(k,\ell)$ tensors. We denote $C^\infty(S^k_\ell TM)$ to be the space of symmetric smooth $(k,\ell)$ tensors.\\

Now assume $(M,g)$ is an $n$ dimensional compact Riemannian manifold with smooth boundary $\partial M$ and the dimension $n\geq 3$. Denote $SM$ for the unit sphere bundle, i.e., $SM=\{(x,\xi)\in TM; |\xi|_g=1\}$.
The mixed ray transform $L_{k,\ell}f$ of $f\in C^\infty(S^k_\ell TM)$ is defined by the formula
\begin{equation}\label{defn_mixedraytransfrom}
(L_{k,\ell}f)(\gamma,\eta)=\int f^{i_1\cdots i_k}_{j_1\cdots j_\ell}(\gamma(t))\eta(t)_{i_1}\cdots\eta(t)_{i_k}\dot{\gamma}(t)^{j_1}\cdots\dot{\gamma}(t)^{j_\ell}\mathrm{d}t,
\end{equation}
where $\gamma(t)$ is a geodesic in $(M,g)$ with two end points on $\partial M$, and $\eta$ is a unit parallel (co)vector field along $\gamma$ and $\eta$ is conormal to $\dot{\gamma}$. To be more precise, $\eta:\gamma\rightarrow T^*M$ such that
\[
\eta(t)\in \dot{\gamma}(t)^\perp=\{\eta\in T_{\gamma(t)}^*M\vert\langle\eta,\dot{\gamma}(t)\rangle=0\},\quad |\eta(t)|_g=1,\quad \eta\text{ is parallel along } \gamma.
\]
We remark here that the definition \eqref{defn_mixedraytransfrom} is different than that given in \cite{Shara} and \cite{de2021generic}, but it is easy to see that they are equivalent. For $\ell=0$, we use the notation $T_k=L_{k,0}$ as in \cite{Shara} to be the so-called transverse ray transform. Explicitly,
\[
(T_kf)(x,\gamma,\eta)=\int f^{i_1\cdots i_k}(\gamma(t))\eta(t)_{i_1}\cdots\eta(t)_{i_k}\mathrm{d}t,
\]

For $k=0$, the geodesic ray transform $I_\ell$ is called the longitudinal ray transform. When the background metric $g$ is Euclidean, $I_0$ is the X-ray transform, which is the basis of \textit{Computerized Tomography} (CT). The case $\ell=1$ comes from the geodesic Doppler transform. The case $\ell=2$ arises from the linearization of the boundary rigidity problem, which reduces to the case $\ell=0$ if restricting to the fixed conformal class. The case $\ell=4$ is related to the travel time tomography problem for quasi-$P$ waves in weakly anisotropic elastic media \cite{vcerveny1982linearized, Shara}.


The invertibility of $I_\ell=L_{0,\ell}$ has been studied quite extensively \cite{Muk2,Muk1,AR,Shara,stefanov2004stability,dairbekov2006integral,paternain2013tensor,paternain2015invariant,stefanov2018inverting} under the assumption that $(M,g)$ is simple. We refer to \cite{de2021generic} for a review. For $k>1$, it is easy to see that potential fields, i.e., fields $f=\mathrm{Sym}\nabla u$ with $u\in C^\infty(S_\ell TM)$, $u\vert_{\partial M}=0$, are in the kernel of $I_\ell$. Here $\mathrm{Sym}$ is the symmetrization operator and $\nabla$ is the Levi-Civita connection induced by $g$.\\

The case $k=1,\ell=0$ corresponds to the \textit{Polarization Tomography} \cite{Shara,Holman}. For $k=2,\ell=2$, the problem arises from the linearization the travel-time tomography problem or inverse boundary value problem for quasi-\textit{S} waves in weakly anisotropic elastic media \cite{Shara, de2021generic}.
We will study the invertibility of $T_1$ and $L_{1,1}$ in this paper, and treat $T_2$ and $L_{2,2}$ in a subsequent paper.

For $n\geq 3$, the kernel of $T_k$ is proved to be trivial when $g=g_0$ is Euclidean, and for $k<n$ when $(M,g)$ is simple and the curvature is close to zero \cite{Shara}. We remark here that when $n=2$, the transverse ray transform $T_k$ has a kernel.\\

The invertibility of the mixed ray transform has a natural obstruction. To characterize the kernel of $L_{k,\ell}$, $k,\ell\geq 1$, we first introduce several operators as in \cite{Shara,de2021generic}. Note that in \cite{Shara,de2021generic}, the tensors are all covariant tensors, but are equivalent to those in the following. Let $\lambda: C^\infty(S^{k-1}_{\ell-1}TM)\rightarrow C^\infty(S^{k}_{\ell}TM)$ be defined by
\[
(\lambda w)^{i_1\cdots i_k}_{j_1\cdots j_\ell}=\mathrm{Sym}(i_1\cdots i_k)\mathrm{Sym}(j_1\cdots j_\ell)(w^{i_1\cdots i_{k-1}}_{j_1\cdots j_{\ell-1}}\delta^{i_k}_{j_\ell}),
\] 
where $\mathrm{Sym}(\,\cdot\,)$ is the symmetrization operator in the indices listed in the argument. The dual of the operator $\lambda$ is the trace operator $\mu:C^\infty(S^{k}_{\ell}TM)\rightarrow C^\infty(S^{k-1}_{\ell-1}TM)$ where
\[
(\mu u)^{i_1\cdots i_{k-1}}_{j_1\cdots j_{\ell-1}}= u^{i_1\cdots i_k}_{j_1\cdots j_\ell}\delta_{i_k}^{j_\ell}.
\]
Now we have
\[
C^\infty(S^{k}_{\ell}TM)=\mathrm{ker}\mu\oplus\mathrm{Im}\lambda.
\]
We denote $\mathcal{B}$ to be the orthogonal projection onto the subspace $C^\infty(\mathcal{B}S^{k}_{\ell}TM):=\mathrm{ker}\mu$. It is clear that tensors in $\mathrm{Im}\lambda$ are in the kernel of $L_{k,\ell}$. Therefore we only need to consider the mixed ray transform $L_{k,\ell}f$ of ``trace-free" $f\in C^\infty(\mathcal{B}S^{k}_{\ell}TM)$.

Now we introduce the operator $\mathrm{d}':C^\infty(\mathcal{B}S^{k}_{\ell-1}TM)\rightarrow C^\infty(S^{k}_{\ell}TM)$ acting on trace-free tensors defined by
\[
(\mathrm{d}'v)^{i_1\cdots i_k}_{j_1\cdots j_\ell}=\mathrm{Sym}(j_1\cdots j_\ell)(v^{i_1\cdots i_k}_{j_1\cdots j_{\ell-1};j_\ell}),
\]
and
\[
\mathrm{d}^\mathcal{B}=\mathcal{B}\mathrm{d}':C^\infty(\mathcal{B}S^{k}_{\ell-1}TM)\rightarrow  C^\infty(\mathcal{B}S^{k}_{\ell}TM).
\]
The adjoint $-\delta^\mathcal{B}:C^\infty(\mathcal{B}S^{k}_{\ell}TM)\rightarrow C^\infty(\mathcal{B}S^{k}_{\ell-1}TM)$ of $\mathrm{d}^\mathcal{B}$ is then given by
\[
(\delta^\mathcal{B}u)^{i_1\cdots i_k}_{j_1\cdots j_{\ell-1};j_\ell}=g^{j_\ell j_{\ell+1}}u^{i_1\cdots i_k}_{j_1\cdots j_\ell;j_{\ell+1}}.
\]
It is now clear that tensor fields $\mathrm{d}^\mathcal{B}v$, with $v\in C^\infty(S^{k}_{\ell-1}TM)$, $v\vert_{\partial M}=0$, are in the kernel of $L_{k,\ell}$. Now, the mixed ray transform $L_{k,\ell}$ is called to be $s$-injective if $L_{k,\ell}f=0$ implies $f=\mathrm{d}^\mathcal{B}v$ for some $v\in C^\infty(\mathcal{B}S^{k}_{\ell-1}TM)$ with $v\vert_{\partial M}=0$.

The problem has been studied under the assumption that the manifold $(M,g)$ is simple. When $n=2$, the $s$-injectivity of $L_{k,\ell}$ for arbitrary $k,\ell\geq 1$ is proved in \cite{de2018mixed}. In higher dimensions ($n\geq 3$), the $s$-injectivity is proved assuming the curvature is close to zero \cite{Shara}. Also, when $n=3$, the $s$-injectivity of $L_{1,1}$ and $L_{2,2}$ is proved to be generically true \cite{de2021generic}. \\

In \cite{UV}, Uhlmann and Vasy proved that, if $n\geq 3$, $\partial M$ is strictly convex and $(M,g)$ admits a strictly convex function (known as the foliation condition), then $I_0$ is invertible. Then the $s$-injectivity of $I_1$ and $I_2$ is proved in \cite{stefanov2018inverting} and $I_4$ in \cite{de2019inverting} under the same condition. The condition was studied extensively in \cite{paternain2019geodesic}. In particular, \cite[Lemma 2.1]{paternain2019geodesic} shows that this condition is satisfied if (1) the manifold has non-negative sectional curvature; or (2) the manifold is simply connected and has non-positive sectional curvature. In this paper, we will prove the injectivity of $T_1$ and the \textit{s}-injectivity of $L_{1,1}$ under this condition, since they are the basic cases for transverse and mixed ray transforms respectively. In a subsequent paper, we will study the injectivity of $T_2$ and the \textit{s}-injectivity of $L_{2,2}$.

The proofs use Melrose's scattering calculus, for which a standard reference is \cite{melrose2020spectral}. The method has also been used to study the boundary rigidity problem \cite{stefanov2021local,stefanov2021local} and the elastic-wave travel time tomography \cite{de2020recovery,zou2019partial}. We also refer to \cite{zhou2018local,zhou2018lens,paternain2019lens} for the study of various geometrical inverse problems in this setting. For further development of this method, we refer to \cite{vasy2020semiclassical,jia2022tensorial,vasy2022x}.\\

The rest of the paper is organized as follows. In Section \ref{section2}, we set up the problems in the setting of Melrose's scattering calculus. In Section \ref{section3}, we prove the invertibility of the transverse ray transform of $1$ tensors, and In Section \ref{section4}, we show the s-injectivity of the mixed ray transform of $1+1$ tensors.

\section{Setting up within the scattering algebra}\label{section2}
 Melrose's scattering algebra uses the compactification of $\mathbb{R}^n$. Using polar coordinates $(r,\theta)$ on $\mathbb{R}^n$, $\mathbb{R}^n\setminus\{0\}$ can be identified with $\{(r,\theta)\vert r\in(0,+\infty),\theta\in\mathbb{S}^{n-1}\}$. Letting $x=r^{-1}$, $\mathbb{R}^n\setminus\{0\}=(0,\infty)_x\times\mathbb{S}^{n-1}_\theta$, and we can add a sphere at the infinity $x=0$ (corresponding to $r=\infty$) to compactify $\mathbb{R}^n$ to $\overline{\mathbb{R}^n}$ such that $\overline{\mathbb{R}^n}\setminus\{0\}$ is identified with $[0,\infty)_x\times\mathbb{S}^{n-1}_\theta$. Note that $\partial \overline{\mathbb{R}^n}=\{x=0\}$, and thus $x=r^{-1}$ is a boundary defining function.

Let $(\widetilde{M},g)$ be a manifold without boundary extending $(M,g)$. So $\gamma\in M$ is extended to a geodesic on $\widetilde{M}$. Let $\rho\in C^\infty(\widetilde{M})$ be a local boundary defining function for $M$, i.e., $\rho> 0$ in $M$, $\rho<0$ on $\widetilde{M}\setminus\overline{M}$, and $\rho=0$ on $\partial M$. We can assume $\rho$ is a smooth function on $\widetilde{M}$. For an arbitrary point $p\in \partial M$, we choose a function $\tilde{x}$ defined on a neighborhood $U$ of $p$ in $\widetilde{M}$, with $\tilde{p}=0$, $\mathrm{d}\tilde{x}(p)=-\mathrm{d}\rho(p)$, $\mathrm{d}\tilde{x}\neq 0$ on $U$, and $\tilde{x}$ has strictly concave level sets from the side of super-level sets. Let $\Omega_c=\{\tilde{x}>-c,\rho\geq 0\}$ such that $\overline{\Omega_c}$ is compact. We will only consider the ray transform along geodesics completely contained in $\Omega_c$

Let $X=\{x\geq 0\}$ with $x=\tilde{x}+c$. The $\partial X=\{x=0\}$ is the boundary of $X$. Now one can add $y^1,\cdots,y^{n-1}$ to $x$ such that $(x,y_1,\cdots,y_{n-1})$ are local coordinates on $X$. To use the Melrose's scattering calculus, one can locally identify $X$ with $\overline{\mathbb{R}^n}$ by identifying $\{x=0\}$ with the boundary of $\overline{\mathbb{R}^n}$ (the `infinity').

Then we introduce the set of scattering vectors $\mathcal{V}_{sc}(X)=x\mathcal{V}_b(X)$, where $\mathcal{V}_b(X)$ is the set of all smooth vector fields tangent to $\partial X$. In the local coordinate chart $(x,y_1,\cdots,y_{n-1})$, a basis for $\mathcal{V}_{sc}(X)$ is $x^2\partial_x,x\partial_{y^1},\cdots, x\partial_{y^{n-1}}$. Denote ${^{sc}TX}$ to be the vector bundle with local basis 
\[
x^2\partial_x,x\partial_{y^1},\cdots, x\partial_{y^{n-1}},
\]
 whose smooth sections are elements in $\mathcal{V}_{sc}(X)$. The dual bundle ${^{sc}T^*X}$ of ${^{sc}TX}$ has a local basis
\[
\frac{\mathrm{d}x}{x^2},\frac{\mathrm{d}y^1}{x},\cdots,\frac{\mathrm{d}y^{n-1}}{x}.
\]
\begin{remark}
In case of $X=\overline{\mathbb{R}^n}$, the basis for ${^{sc}T}\overline{\mathbb{R}^n}$ becomes
\[
x^2\partial_x=-\partial_r,\,r^{-1}\partial_{\theta^1},\cdots,r^{-1}\partial_{\theta^{n-1}},
\]
and the basis for ${^{sc}T}^*\overline{\mathbb{R}^n}$ is
\[
-\mathrm{d}r,\,r\mathrm{d}\theta^1,\cdots,r\mathrm{d}\theta^{n-1}
\]
where $r=\frac{1}{x}$ and $\theta_j$ are local coordinates on the unit sphere $\mathbb{S}^{n-1}$. 
\end{remark}

Denote $S^k_\ell{^{sc}T}X$ to be the subbundle of $S^k_\ell TX$, for which
\[
\begin{split}
&\underbrace{(x^2\partial_x)\otimes^s\cdots\otimes^s(x^2\partial_x)}_{k-m}\otimes^s \underbrace{(x\partial_{y^{i_1}})\otimes^s\cdots\otimes^s(x\partial_{y^{i_m}})}_{m}\\
&\quad\quad\quad\quad\otimes \underbrace{(\frac{\mathrm{d}x}{x^2})\otimes^s\cdots\otimes^s(\frac{\mathrm{d}x}{x^2})}_{\ell-p}\otimes^s \underbrace{(\frac{\mathrm{d}y^{j_1}}{x})\otimes^s\cdots\otimes^s(\frac{\mathrm{d}y^{j_p}}{x})}_{p},\quad 0\leq m\leq k, 0\leq p\leq \ell,
\end{split}
\]
is a local basis, where $\otimes^s$ is the symmetric direct product.
We can take a scattering metic $g_{sc}$, which is a positive definite smooth section of $S_2 {^{sc}T}^*X$, is of the form $g_{sc}=\frac{\mathrm{d}x^2}{x^4}+\frac{h(x,y)}{x^2}$, where $h$ is a metric on the level sets of $x$.

Denote $L^2_{sc}(X)$ to be the $L^2$ space given by identification by $\overline{\mathbb{R}^n}$. Then the measure on $X$ is a nondegenerate positive multiple of $x^{-n-1}\mathrm{d}x\mathrm{d}y$. The weighted Sobolev space is defined by $u\in H^{m,l}_{sc}(X)$ if and only if $x^{-\ell}V_1\cdots V_ku\in L^2_{sc}(X)$ for all $k\leq m$ and $V_j\in\mathcal{V}_{sc}(X)$. \\
%

In this paper, we will consider the invertibility of the local transverse and mixed ray transform. Therefore, in the following we only need to consider the geodesics in $M$ which are completely contained in $X$. 
\subsection{Redefinition of the mixed ray transforms}
First, we need to define an equivalent version of the mixed ray transform so it could fit into the setting of scattering calculus. We emphasize here that the version used in \cite{Shara,de2018mixed} does not really work in this setting.\\

Now assume $f\in S^k_\ell T_zX$.
For any $w\in T_zX$, define
\[
\Lambda_w:S^k_\ell T_zX\rightarrow S^{k}T_zX,\quad (\Lambda_wf)^{i_1\cdots i_k}=f^{i_1\cdots i_k}_{j_1\cdots j_\ell}w^{j_1}\cdots w^{j_\ell},
\]
Take $v\in T_z^*X$, and denote $v^{\perp}=\{\eta\in T_zX|\langle \eta,v\rangle=0\}$.
Define
\[
p_{w,v}:T_zX\rightarrow v^{\perp},\quad (p_{w,v}\zeta)^i=(p_{w,v})^i_j\zeta^j=\zeta^i-\frac{w^iv_j}{\langle w,v\rangle}\zeta^j.
\]
Then define
\[
P_{w,v}:S^{k}T_zX\rightarrow S^{k}T_zX,\quad (P_{w,v}f)^{i_1\cdots i_k}=f^{m_1\cdots m_k}(p_{w,v})_{m_1}^{i_1}\cdots (p_{w,v})_{m_k}^{i_k}.
\]

For any $z\in M$, and $\zeta\in T_zM$, we let $\gamma$ to be the geodesic such that
\begin{equation}\label{geodesic_initial}
\gamma(0)=z,\quad \dot{\gamma}(0)=\zeta.
\end{equation}
Denote $\mathcal{T}_\gamma^{s,t}$ to be parallel transport (with respect to $g$) from $\gamma(t)$ to $\gamma(s)$.
Let $\eta(t)\in T^*M$ be a covector field along $\gamma$ such that $\eta(t)$ is conormal to $\dot{\gamma}(t)$ and $\eta$ is parallel (w.r.t. $g$) along $\gamma$. 
Take $\vartheta\in T^*_zM$ such that $\langle \vartheta,\zeta\rangle\neq 0$, and $v=a\vartheta+\eta$, and $\vartheta(t)$ be the parallel transport of $\vartheta$ from $\gamma(0)$ to $\gamma(t)$.
We verify that
\[
\begin{split}
&\langle\mathcal{T}_\gamma^{0,t}P_{\dot{\gamma}(t),\vartheta(t)}\Lambda_{\dot{\gamma}(t)}f(\gamma(t)),v^k\rangle\\
=&\langle P_{\dot{\gamma}(t),\vartheta(t)}\Lambda_{\dot{\gamma}(t)}f(\gamma(t)),(\mathcal{T}_\gamma^{t,0}v)^k\rangle\\
=& \left(\delta_{m_1}^{i_1}-\frac{\dot{\gamma}^{i_1}(t)\vartheta_{m_1}(t)}{\langle\dot{\gamma}(t),\vartheta(t)\rangle}\right)\cdots \left(\delta_{m_k}^{i_k}-\frac{\dot{\gamma}^{i_k}(t)\vartheta_{m_k}(t)}{\langle\dot{\gamma}(t),\vartheta(t)\rangle}\right)f^{m_1\cdots m_k}_{j_1\cdots j_\ell}(\gamma(t))\dot{\gamma}^{j_1}(t)\cdots\dot{\gamma}^{j_\ell}(t)\\
&\quad\quad\quad\quad(a\vartheta_{i_1}(t)+\eta_{i_1}(t))\cdots (a\vartheta_{i_k}(t)+\eta_{i_k}(t))\\
=&f^{i_1\cdots i_k}_{j_1\cdots j_\ell}(\gamma(t))\dot{\gamma}^{j_1}(t)\cdots\dot{\gamma}^{j_\ell}(t)\eta_{i_1}(t)\cdots \eta_{i_k}(t),
\end{split}
\]
where $v^k=(\underbrace{v,\cdots,v}_k)$.
Then it is easy to see that
\[
(L_{k,\ell}f)(\gamma,\eta)=\int\langle\mathcal{T}_\gamma^{0,t}P_{\dot{\gamma}(t),\vartheta(t)}\Lambda_{\dot{\gamma}(t)}f(\gamma(t)),v^k\rangle\mathrm{d}t.
\]
Since the choice of $v^k=(a\vartheta+\eta)^k$ span the whole space of $S^kT^*_zM$,
we can redefine the local mixed ray transform as follows
\[
L_{k,\ell}f(z,\zeta,\vartheta)=\int\mathcal{T}_\gamma^{0,t}P_{\dot{\gamma}(t),\vartheta(t)}\Lambda_{\dot{\gamma}(t)}f(\gamma(t))\mathrm{d}t,
\]
where $\gamma$ is a geodesic completely in $X$ satisfying the initial condition \eqref{geodesic_initial}.

\subsection{Geodesics near the lifted diagonal}
We use local coordinates $z=(x,y)$, where $x$ is the boundary defining function. Elements in $T_pM$ can be written as $\lambda\partial_x+\omega\partial_y$. The unit speed geodesics almost tangential to level sets of $x$ through a point $p=(x,y)$ can be parametrized by $(\lambda,\omega)$, where $\omega$ is of unit length (with respect to the Euclidean metric).

As mentioned the \cite{UV}, the maps 
\[
\Gamma_+:S\widetilde{M}\times[0,\infty)\rightarrow [\widetilde{M}\times\widetilde{M};\mathrm{diag}],\quad \Gamma_+(x,y,\lambda,\omega,t)=((x,y),\gamma_{x,y,\lambda,\omega}(t))
\]
and
\[
\Gamma_-:S\widetilde{M}\times(-\infty,0]\rightarrow [\widetilde{M}\times\widetilde{M};\mathrm{diag}],\quad \Gamma_-(x,y,\lambda,\omega,t)=((x,y),\gamma_{x,y,\lambda,\omega}(t))
\]
and diffeomorphisms near $S\widetilde{M}\times\{0\}$.
Here $[\widetilde{M}\times\widetilde{M};\mathrm{diag}]$ is the blow-up of $\widetilde{M}$ at the diagonal $z=z'$.

Writing the local coordinates $(x,y,x',y')$ for $\widetilde{M}\times\widetilde{M}$, in the region $|x-x'|<C|y-y'|$, we use the coordinates on $[\widetilde{M}\times\widetilde{M};\mathrm{diag}]$ as
\[
x,y,|y'-y|,\frac{x'-x}{|y'-y|},\frac{y'-y}{|y'-y|}.
\]
Notice that these are
\[
x,y,x|Y|,\frac{xX}{|Y|},\hat{Y},
\]
where
\[
x,y,X=\frac{x'-x}{x^2},Y=\frac{y'-y}{x}
\]
are local coordinates on Melrose's scattering double space near the lifted scattering diagonal.

%

Denote
\[
(x,y,\lambda,\omega,t)=\Gamma_\pm^{-1}\left(x,y,x|Y|,\frac{xX}{|Y|},\hat{Y}\right).
\]
By the discussion in \cite{UV,stefanov2018inverting,paternain2019geodesic}, $\lambda,\omega,t$ can be written in terms of $x,y,x|Y|,\frac{xX}{|Y|},\hat{Y}$ as
\begin{equation}\label{lambdaomega}
\begin{split}
\lambda&=(\Lambda\circ\Gamma_\pm^{-1})(x,y,x|Y|,\frac{xX}{|Y|},\hat{Y})\\
&=\pm x\frac{X-\alpha(x,y,x|Y|,\frac{xX}{|Y|},\hat{Y})|Y|^2}{|Y|}+x^2\tilde{\Lambda}_{\pm}\left(x,y,x|Y|,\frac{xX}{|Y|},\hat{Y}\right),\\
\omega&=(\Omega\circ\Gamma_{\pm}^{-1})\left(x,y,x|Y|,\frac{xX}{|Y|},\hat{Y}\right)=\pm \hat{Y}+x|Y|\tilde{\Omega}_\pm\left(x,y,x|Y|,\frac{xX}{|Y|},\hat{Y}\right),
\end{split}
\end{equation}
and
\begin{equation}\label{tpm}
 t=(T\circ\Gamma^{-1}_\pm)\left(x,y,x|Y|,\frac{xX}{|Y|},\hat{Y}\right)=\pm x|Y|+x^2|Y|^2\tilde{T}_\pm\left(x,y,x|Y|,\frac{xX}{|Y|},\hat{Y}\right),
\end{equation}
where $\tilde{\Lambda}_{\pm},\tilde{\Omega}_\pm,\tilde{T}_\pm$ are smooth functions in terms of $x,y,x|Y|,\frac{xX}{|Y|},\hat{Y}$.

Assume $(\gamma(t),\gamma'(t))=(x',y',\lambda',\omega')$. As in \cite{UV}, we can choose coordinates $x$ and $y^j$ such that $\partial_{y^j}$ and $\partial_x$ are orthogonal., i.e., the metric is of the form $f(x,y)\mathrm{d}x^2+k(x,y)\mathrm{d}y^2$, near $\{x=0\}$. Note that this choice of local coordinates is also crucial in \cite{stefanov2021local}. By the geodesic equation, we have
\[
\ddot{\gamma}^x(0)=\dot{\gamma}^{y^i}(0)\dot{\gamma}^{y^j}(0)\Gamma_{y^iy^j}^x(\gamma(0))+\mathcal{O}(x^2).
\]
Denote
\[
\dot{\gamma}^{y^i}(0)\dot{\gamma}^{y^j}(0)\Gamma^x_{y^iy^j}=-\frac{1}{2}g^{xx}\partial_xk_{y^iy^j}\dot{\gamma}^{y^i}(0)\dot{\gamma}^{y^j}(0)=H_{ij}(x,y)\omega^i\omega^j,
\]
where
\[
H_{ij}(x,y)=-\frac{1}{2}g^{xx}\partial_xk_{y^iy^j}(x,y).
\]
Thus
\[
x'=x+\lambda t+\frac{1}{2}(H_{ij}\omega^i\omega^j)t^2+\mathcal{O}(t^3).
\]
Then we can see that (cf. also \cite[Section 3.1]{UV})
\[
\alpha(x,y,0,0,\omega)=\frac{1}{2}H_{ij}(x,y)\omega^i\omega^j,
\]
and $\lambda',\omega'$ can be expressed in terms of $x,y,x|Y|,\frac{xX}{|Y|},\hat{Y}$ as (cf. \cite{stefanov2018inverting})
\[
\begin{split}
\Lambda'\circ\Gamma^{-1}_\pm&=\pm x\frac{X+\alpha(x,y,x|Y|,\frac{xX}{|Y|},\hat{Y})|Y|^2}{|Y|}+x^2|Y|^2\tilde{\Lambda}_\pm'\left(x,y,x|Y|,\frac{xX}{|Y|},\hat{Y}\right),\\
\Omega'\circ\Gamma^{-1}_\pm&=\pm \hat{Y}+x^2|Y|^2\tilde{\Omega}_\pm'\left(x,y,x|Y|,\frac{xX}{|Y|},\hat{Y}\right).
\end{split}
\]

\section{Transverse ray transform of $1+0$ tensors}\label{section3}
In this section, we focus on the case $k=1,\ell=0$, that is the transverse ray transform of $(1,0)$-tensors.\\

Let $(x,y,\lambda,\omega)\in\mathbb{R}\times\mathbb{R}^{n-1}\times\mathbb{R}\otimes\mathbb{S}^{n-2}$.
We write down the local transverse transform for a $(1,0)$ tensor $f$ explicitly:
\[
\begin{split}
T_1f(x,y,\lambda,\omega)=&\int_{\mathbb{R}}\mathcal{T}_{\gamma_{x,y,\lambda,\omega}}^{0,t}P_{\dot{\gamma}_{x,y,\lambda,\omega}(t),\vartheta_{x,y,\omega}(t)}f({\gamma_{x,y,\lambda,\omega}(t)})\mathrm{d}t\\
=&\int_{\mathbb{R}}\mathcal{T}_{\gamma_{x,y,\lambda,\omega}}^{0,t} \left(\mathrm{Id}-\frac{\dot{\gamma}_{x,y,\lambda,\omega}(t)\langle\vartheta_{x,y,\omega}(t),\,\cdot\,\rangle}{\langle\dot{\gamma}_{x,y,\lambda,\omega}(t),\vartheta_{x,y,\omega}(t)\rangle}\right)f({\gamma_{x,y,\lambda,\omega}(t)})\mathrm{d}t,
\end{split}
\]
where $\gamma_{x,y,\lambda,\omega}$ is a geodesic in $(X\cap \widetilde{M},g)$ such that
\[
\gamma_{x,y,\lambda,\omega}(0)=(x,y),\quad\dot{\gamma}_{x,y,\lambda,\omega}(0)=(\lambda,\omega),
\]
$\vartheta_{x,y,\omega}(t)$ is parallel along $\gamma_{x,y,\lambda,\omega}$ with 
\begin{equation}\label{vartheta_initial}
\vartheta=\vartheta_{x,y,\omega}(0)=(\vartheta_x(0),\vartheta_y(0))=h(\omega)\mathrm{d}y,
\end{equation}
and $\mathcal{T}_{\gamma_{x,y,\lambda,\omega}}^{0,t}$ is a parallel transport (w.r.t. $g$) from $\gamma_{x,y,\lambda,\omega}(t)$ to $\gamma_{x,y,\lambda,\omega}(0)$. Note that
\[
\langle\vartheta,\zeta\rangle=h(\omega,\omega)\neq 0.
\]

If $|\lambda|\leq C\sqrt{x}$ for a sufficiently small $C$, $\gamma_{x,y,\lambda,\omega}$ stays in $X$ as long as it is in $M$.
Then we take $\chi$ to be a smooth, even, non-negative function with compact support, and define
\[
Lv(x,y)=x^{-2}\int\left(\mathrm{Id}-\frac{\omega\partial_y\otimes g_{sc}(\lambda\partial_x+\omega\partial_y)}{h(\omega\partial_y,\omega\partial_y)}\right)\chi(\lambda/x)v(\gamma_{x,y,\lambda,\omega}(0),\dot{\gamma}_{x,y,\lambda,\omega}(0))\mathrm{d}\lambda\mathrm{d}\omega.
\]
Then 
\[
(LT_1f)(x,y)=x^{-2}\int\left(\mathrm{Id}-\frac{\omega\partial_y\otimes g_{sc}(\lambda\partial_x+\omega\partial_y)}{h(\omega\partial_y,\omega\partial_y)}\right)\chi(\lambda/x)T_1f(x,y,\lambda,\omega)\mathrm{d}\lambda\mathrm{d}\omega.
\]
We will show that the operator $N_\mathsf{F}=e^{-\mathsf{F}/x}LT_1e^{\mathsf{F}/x}$ is elliptic as a scattering pseudodifferential operator, and invertible between proper spaces of functions supported near $x=0$.\\

Recall that (cf. \cite[(3.15)]{UV})
\[
(\Gamma^{-1}_\pm)^*\mathrm{d}t\mathrm{d}\lambda\mathrm{d}\omega=J_\pm\left(x,y,x|Y|,\frac{xX}{|Y|},\hat{Y})\right)x^2|Y|^{-n+1}\mathrm{d}X\mathrm{d}Y,
\]
where $J_\pm$ is smooth and positive, and $J_\pm|_{x=0}=1$.\\


Assume that $\vartheta=(\vartheta_x,\vartheta_y)$ are given in terms of $x,x',y,y'$ as
\[
\begin{split}
\vartheta_x(x,y,\lambda,\omega,t)=\Xi_x'\circ\Gamma^{-1}_\pm\left(x,y,x|Y|,\frac{xX}{|Y|},\hat{Y}\right),\\
\vartheta_y(x,y,\lambda,\omega,t)=\Xi_y'\circ\Gamma^{-1}_\pm\left(x,y,x|Y|,\frac{xX}{|Y|},\hat{Y}\right).
\end{split}
\]
%
Since the covector field $\vartheta$ is parallel along $\gamma$, we have
\[
\dot{\vartheta}_k(0)=-\vartheta^j(0)\dot{\gamma}^i(0)\Gamma_{ij,k}(\gamma(0)),
\]
where $\Gamma_{ij,k}$ is the Christoffel symbol of $g$ under coordinates $(x,y)$.
By this equation and \eqref{vartheta_initial}, we have
\[
\dot{\vartheta}_x(0)=\mathcal{O}(x^{-2}), \quad\dot{\vartheta}_{y^i}(0)+\mathcal{O}(x^{-2}).
\]
Therefore, using \eqref{lambdaomega} and \eqref{tpm}, we have
\[
\vartheta_{x}(t)=\mathcal{O}(x^{-1}),\quad \vartheta_{y}(t)=h(\omega)+\mathcal{O}(x).
\]
So, setting $x=0$, we see that
\[
\begin{split}
(\Xi_x'\circ\Gamma^{-1}_{\pm})(0,y,0,0,\hat{Y})&=0,\\
(\Xi_y'\circ\Gamma^{-1}_{\pm})(0,y,0,0,\hat{Y})&=h(\hat{Y}).
\end{split}
\]

Next let us examine how the parallel transport $\mathcal{T}_{\gamma_{x,y,\lambda,\omega}}^{0,t}$ acts on scattering co-vector fields.
Assume that
\[
V=v^xx^2\partial_x+v^{y^i}x\partial_{y^i},
\]
is a parallel vector field along $\gamma$.
Then
\[
\dot{V}^x(0)=V^{y^i}(0)\dot{\gamma}^{y^i}(0)\Gamma^x_{y^iy^j}+\mathcal{O}(x^2)=-V^{y^i}\omega^j\frac{1}{2}g^{xx}(-\partial_xk_{ij})+\mathcal{O}(x^2),
\]
and thus
\[
\begin{split}
V^x(t)=&V^x+\frac{1}{2}g^{xx}\partial_xk_{ij}\omega^jV^{y^i}t+\mathcal{O}(x^3)\\
=&(v^x-|Y|\langle H\hat{Y},v^y\rangle)x^2+\mathcal{O}(x^3).
\end{split}
\]
By similar consideration, we obtain
\[
V^{y^i}(t)=v^ix+\mathcal{O}(x^2).
\]
To summarize, we have
\[
\left(\begin{array}{c}
v^x(t)\\
v^y(t)
\end{array}\right)=\left(\begin{array}{cc}
1 &-|Y|\langle H\hat{Y},\,\cdot\,\rangle\\
0 &\mathrm{Id}
\end{array}\right)\left(\begin{array}{c}
v^x(0)\\
v^y(0)
\end{array}\right)+\mathcal{O}(x).
\]
We can take local coordinates $y^1,\cdots, y^{n-1}$ such that at the point of interest $(0,y)$, $\alpha(0,y,0,0,\hat{Y})=\alpha(y)$ is independent of $\hat{Y}$. Then $H(0,y)=\alpha(y)I$ and
\[
H(0,y)\hat{Y}=\alpha(y)\hat{Y}.
\]
This shall be related to the choice of $\vartheta$.



The Schwartz kernel of $N_{\mathsf{F}}=e^{-\mathsf{F}/x}LT_1e^{\mathsf{F}/x}$, for $\mathsf{F}>0$, is then
\[
\begin{split}
&K^\flat(x,y,X,Y)\\
=&\sum_{\pm}e^{-\mathsf{F}X/(1+xX)}\chi\left(\frac{X-\alpha(x,y,x|Y|,\frac{xX}{|Y|},\hat{Y})}{|Y|}+x\tilde{\Lambda}_\pm\left(x,y,x|Y|,\frac{xX}{|Y|},\hat{Y})\right)\right)\\
&\left(\mathrm{Id}-\frac{\left((\Omega\circ\Gamma^{-1}_{\pm})x\partial_y\right)\otimes\left(x^{-1}(\Lambda\circ\Gamma^{-1}_{\pm})x^2\partial_x+(\Omega\circ\Gamma^{-1}_{\pm})\frac{h(\partial_y)}{x}\right)}{|\Omega\circ\Gamma^{-1}_{\pm}|^2h(\partial_y,\partial_y)}\right)\mathcal{T}_\pm(x,y,X, Y)\\
&\left(\mathrm{Id}-\frac{\left(x^{-1}(\Lambda'\circ\Gamma^{-1}_{\pm})x^2\partial_x+(\Omega'\circ\Gamma^{-1}_{\pm})x\partial_y\right)\otimes\left(x^{-1}(\Xi_x'\circ\Gamma^{-1}_{\pm})\frac{\mathrm{d}x}{x^2}+(\Xi_y'\circ\Gamma^{-1}_{\pm})\frac{h(\partial_y)}{x}\right)}{x^{-2}(\Lambda'\circ\Gamma^{-1}_{\pm})(\Xi_x'\circ\Gamma_{\pm}^{-1})+(\Omega'\circ\Gamma^{-1}_{\pm})(\Xi_y'\circ\Gamma^{-1}_{\pm})h(\partial_y,\partial_y)}\right)\\
&|Y|^{-n+1}J_\pm\left(x,y,x|Y|,\frac{xX}{|Y|},\hat{Y})\right),
\end{split}
\]
where 
\[
\mathcal{T}_\pm(0,y,X,Y)=\left(\begin{array}{cc}
1 &-\alpha|Y|\langle \hat{Y},\,\cdot\,\rangle\\
0 &\mathrm{Id}
\end{array}\right)^{-1}=\left(\begin{array}{cc}
1 &\alpha|Y|\langle \hat{Y},\,\cdot\,\rangle\\
0 &\mathrm{Id}
\end{array}\right).
\]

\begin{remark}
If $\lambda/x\in\mathrm{supp}\chi$, $\gamma_{x,y,\lambda,\omega}$ stays in $X$ as long as it is in $\widetilde{M}$ for sufficiently small $x$. Then the operator $N_\mathsf{F}\in\Psi_{sc}^{-1,0}(X)$ is a well defined pseudodifferential operator on $O\subset X\cap \widetilde{M}$ where $x$ is small enough in $O$. Here $\Psi_{sc}$ stands for the scattering calculus of Melrose.
\end{remark}

We will prove that $N_\mathsf{F}$ is \textit{fully elliptic} as a scattering pseudodifferential operator for $x$ sufficiently close to $0$. For this, we need to show that the principal symbol at $(z,\zeta)=(x,y,\xi,\eta)\in {^{sc}T^*X}$ is positive definite at both the \textit{fiber infinity} (when $|\zeta|\rightarrow+\infty$) of ${^{sc}T^*X}$ and the \textit{base infinity} (when $x=0$) of ${^{sc}T^*X}$. Here $(\xi,\eta)$ is the Fourier dual variables of $(X,Y)$. For the principal symbol at the fiber infinity, we also only need to analyze the principal symbol at $x=0$, since the Schwartz kernel $K^\flat$ is smooth in $(x,y)$ down to $x=0$.

In the following, we denote the principal symbol of a scattering pseudodifferential operator $A$ by $\sigma(A)$, the principal symbol at fiber infinity by $\sigma_p(A)$, and the principal symbol at base infinity by $\sigma_{sc}(A)$.

\begin{lemma}\label{ellipticity10fiber}
The operator $N_\mathsf{F}$ is elliptic at fiber infinity of $^{sc}T^*X$ in $\widetilde{M}$.
\end{lemma}
\begin{proof}
We only need to analyze the principal symbol at $x=0$.
We can take $g_{sc}=\frac{\mathrm{d}x^2}{x^4}+\frac{\mathrm{d}y^2}{x^2}$ ($h_{ij}=\delta_{ij}$) at the point of interest.
Writing
\[
S=\frac{X-\alpha(\hat{Y})|Y|^2}{|Y|},\quad \hat{Y}=\frac{Y}{|Y|}.
\]
At the scattering front face $x=0$, the kernel of $N_\mathsf{F}$ is
\[
\begin{split}
K^\flat(0,y,X,Y)&=e^{-\mathsf{F}X}|Y|^{-n+1}\chi(S)\left(\mathrm{Id}-\left(\hat{Y}\cdot(x\partial_y)\right)\left(S\frac{\mathrm{d}x}{x^2}+\hat{Y}\cdot\frac{\mathrm{d}y}{x}\right)\right)\\
&\mathcal{T}(0,y,X,Y)\left(\mathrm{Id}-\left((S+2\alpha |Y|)(x^2\partial_x)+\hat{Y}\cdot(x\partial_y)\right)\left(\hat{Y}\cdot\frac{\mathrm{d}y}{x}\right)\right)
\end{split}
\]


Using the basis 
\[
x^2\partial_x,x\partial_y
\]
for $1$-tensors, we can write $K^\flat(0,y,X,Y)$ in matrix form
\begin{equation}\label{schwartkernelT1}
\begin{split}
K^\flat(0,y,X,Y)=&e^{-\mathsf{F}X}|Y|^{-n+1}\chi(S)\left(
\begin{array}{cc}
1 & 0\\
-S\hat{Y}& I-\hat{Y}\langle\hat{Y},\,\cdot\,\rangle
\end{array}
\right)\left(\begin{array}{cc}
1 &\alpha |Y|\langle \hat{Y},\,\cdot\,\rangle\\
0 &I
\end{array}\right)\\
&\quad\quad\quad\quad\quad\quad\quad\quad\left(
\begin{array}{cc}
1 & -(S+2\alpha|Y|)\langle \hat{Y},\,\cdot\,\rangle\\
0&  I-\hat{Y}\langle\hat{Y},\,\cdot\,\rangle
\end{array}
\right)
\\
=&e^{-\mathsf{F}X}|Y|^{-n+1}\chi(S)\left(
\begin{array}{cc}
1 & -(S+2\alpha|Y|)\langle \hat{Y},\,\cdot\,\rangle\\
-S\hat{Y} & I+(S(S+2\alpha|Y|)-1)\hat{Y}\langle\hat{Y},\,\cdot\,\rangle
\end{array}
\right).
\end{split}
\end{equation}

The principal symbol $\sigma_p(N_\mathsf{F})(0,y,\xi,\eta)$ of $K^\flat(0,y,X,Y)$ at fiber infinity can be obtained by integrating the restriction of the Schwartz kernel to the front face, $|Y|=0$, after removing the singular factor $|Y|^{-n+1}$, along the $(\xi,\eta)$-equatorial sphere. 
Therefore, we only need to integrate
\begin{equation}\label{T10_matrix}
\chi(\tilde{S})\left(
\begin{array}{cc}
1 & -\tilde{S}\langle \hat{Y},\,\cdot\,\rangle\\
-\tilde{S}\hat{Y} & I+(\tilde{S}^2-1)\hat{Y}\langle\hat{Y},\,\cdot\,\rangle
\end{array}
\right).\end{equation}
with $\tilde{S}=\frac{X}{|Y|}$.
We need to integrate
with $(\tilde{S},\hat{Y})$ running through the $(\xi,\eta)$-equatorial sphere $\xi\tilde{S}+\eta\cdot\hat{Y}=0$. For $\chi\geq 0$, this matrix is a positive semidefinite.

In the following, we take $\tilde{S}$ sufficiently close to $0$ so that $\chi(\tilde{S})=1$, and thus only need to consider the matrix
 \[
 \left(
\begin{array}{cc}
1 & -\tilde{S}\langle \hat{Y},\,\cdot\,\rangle\\
-\tilde{S}\hat{Y} & I+(\tilde{S}^2-1)\hat{Y}\langle\hat{Y},\,\cdot\,\rangle
\end{array}
\right)
 \]
For any $v=(v_0,v')$ with $v'\in\mathbb{R}^{n-1}$, take $(\tilde{S},\hat{Y})$ satisfying $\xi\tilde{S}+\eta\cdot\hat{Y}=0$. If $v=(v_0,v')$ belongs to the kernel to the above matrix, then
\begin{equation}\label{eq_v0vp}
\begin{split}
v_0-\tilde{S}\langle\hat{Y},v'\rangle =0,\\
-\tilde{S}v_0\hat{Y}+v'+(\tilde{S}^2-1)\langle\hat{Y},v'\rangle\hat{Y}=0.
\end{split}
\end{equation}
Taking $\tilde{S}=0$ and a unit vector $\hat{Y}=\hat{\eta}^\perp$ orthogonal to $\eta$, we can conclude that $v_0=0$ and $v'=\langle\hat{\eta}^\perp,v'\rangle\hat{\eta}^\perp$. If $\eta=0$, $\hat{\eta}^\perp$ can be any unit vector, and so we can already conclude that $v'=0$. Now we assume that $\eta\neq 0$. If $\xi\neq 0$, we take $\hat{Y}=\epsilon\hat{\eta}+\sqrt{1-\epsilon^2}\hat{\eta}^\perp$ and $\tilde{S}=-\frac{\epsilon|\eta|}{\xi}$ with $\epsilon$ small enough; if $\xi=0$, we take $\hat{Y}=\hat{\eta}^\perp$ and an arbitrary $\tilde{S}\neq 0$ small enough. Then we get from the first equation in \eqref{eq_v0vp} that
\[
\langle \hat{Y},v'\rangle=0.
\]
Then $v'=0$ by the second equation in \eqref{eq_v0vp}. Combined with the result $v_0=0$, we conclude $v=0$.

 Therefore the integral of the matrix \eqref{T10_matrix} along the equatorial sphere is positive definite. So the boundary principal symbol of $N_\mathsf{F}$ is elliptic at fiber infinity, and by continuity, the principal symbol of $N_\mathsf{F}$ is elliptic at fiber inifinity when $x$ is sufficiently close to $0$.
\end{proof}

\begin{lemma}
The operator $N_\mathsf{F}$ is elliptic at finite points of $^{sc}T^*X$ in $\widetilde{M}$ for $\mathsf{F}>0$ sufficiently large.
\end{lemma}
\begin{proof}
We calculate the $(X,Y)$-Fourier transform of the kernel on the front face $x=0$
\[
\begin{split}
K^\flat(0,y,X,Y)=&e^{-\mathsf{F}X}\chi(S)|Y|^{-n+1}\left(
\begin{array}{cc}
1 & -(S+2\alpha|Y|)\langle \hat{Y},\,\cdot\,\rangle\\
-S\hat{Y} & I+(S(S+2\alpha|Y|)-1)\hat{Y}\langle\hat{Y},\,\cdot\,\rangle
\end{array}
\right).
\end{split}
\]

Recall that $\chi$ is chosen to be a compactly supported function. However, as in \cite{stefanov2018inverting}, we can take a Gaussian $\chi_0$ to do the actual computation, and then approximate $\chi$ by a sequence $\chi_k=\phi(\cdot/k)\chi_0$, where $\phi\in C_c^\infty(\mathbb{R})$, $\phi\geq 0$.

Taking $\chi(s)=e^{-s^2/(2\nu(\hat{Y}))}$, so $\hat{\chi}(\cdot)=c\sqrt{\nu}e^{-\nu|\cdot|^2/2}$, we get the $X$-Fourier transform of $K^\flat(0,y,X,Y)$
\[
\begin{split}
&\mathcal{F}_XK^\flat(0,y,\xi,Y)\\
=&|Y|^{2-n}e^{-\mathrm{i}\alpha(-\xi-\mathrm{i}\mathsf{F})|Y|^2}\left(
\begin{array}{cc}
1& -(-D_\sigma+2\alpha|Y|)\langle \hat{Y},\,\cdot\,\rangle\\
-(-D_\sigma)\hat{Y} & \mathrm{Id}+((-D_\sigma)(-D_\sigma+2\alpha|Y|)-1)\hat{Y}\langle\hat{Y},\,\cdot\,\rangle
\end{array}
\right)\hat{\chi}((-\xi-\mathrm{i}\mathsf{F})|Y|)\\
=&c\sqrt{\nu}|Y|^{2-n}e^{\mathrm{i}\alpha(\xi+\mathrm{i}\mathsf{F})|Y|^2}\left(
\begin{array}{cc}
1& -(-D_\sigma+2\alpha|Y|)\langle \hat{Y},\,\cdot\,\rangle\\
-(-D_\sigma)\hat{Y} & \mathrm{Id}+((-D_\sigma)(-D_\sigma+2\alpha|Y|)-1)\hat{Y}\langle\hat{Y},\,\cdot\,\rangle
\end{array}
\right) e^{-\nu(\xi+\mathrm{i}\mathsf{F})^2|Y|^2/2},
\end{split}
\]
with some positive constant $c$, and $D_\sigma$ differentiating the argument of $\hat{\chi}$.

We take $\nu=\mathsf{F}^{-1}\alpha$, then
\[
e^{\mathrm{i}\alpha(\xi+\mathrm{i}\mathsf{F})|Y|^2}e^{-\nu(\xi+\mathrm{i}\mathsf{F})^2|Y|^2/2}=e^{-\nu(\xi^2+\mathsf{F}^2)|Y|^2/2}.
\]
We denote
\[
\phi(\xi,\hat{Y})=\nu(\xi^2+\mathsf{F}^2)
\]
in the following.
For computing the $Y$-Fourier transform, we only need to compute the following integral
\[
\begin{split}
\int_{\mathbb{S}^{n-2}}\int_0^\infty &\sqrt{\nu}|Y|^{2-n}e^{\mathrm{i}|Y|\hat{Y}\cdot\eta}\left(
\begin{array}{cc}
1& -(-D_\sigma+2\alpha|Y|)\langle \hat{Y},\,\cdot\,\rangle\\
-(-D_\sigma)\hat{Y} & \mathrm{Id}+((-D_\sigma)(-D_\sigma+2\alpha|Y|)-1)\hat{Y}\langle\hat{Y},\,\cdot\,\rangle
\end{array}
\right)\\
&e^{-\nu(\xi+\mathrm{i}\mathsf{F})^2|Y|^2/2}|Y|^{n-2}\mathrm{d}|Y|\mathrm{d}\hat{Y},
\end{split}
\]
where $D_\sigma$ is differentiating the argument of $\hat{\chi}$.

Consider $h=\mathsf{F}^{-1}$ as a semiclassical parameter, and rescale $(\xi_\mathsf{F},\eta_\mathsf{F})=\mathsf{F}^{-1}(\xi,\eta)$.
The above integral can be written as
\[
\begin{split}
&\int_{\mathbb{S}^{n-2}}\int_0^\infty \sqrt{\nu}e^{\mathrm{i}|Y|\hat{Y}\cdot\eta}\left(
\begin{array}{cc}
1&-(\mathrm{i}\nu(\xi+\mathrm{i}\mathsf{F})+2\alpha)|Y|\langle \hat{Y},\,\cdot\,\rangle\\
-(\mathrm{i}\nu(\xi+\mathrm{i}\mathsf{F}))|Y|\hat{Y} & \mathrm{Id}+((\mathrm{i}\nu(\xi+\mathrm{i}\mathsf{F}))(\mathrm{i}\nu(\xi+\mathrm{i}\mathsf{F})+2\alpha)|Y|^2-1)\hat{Y}\langle\hat{Y},\,\cdot\,\rangle
\end{array}
\right)\\
&\quad\quad\quad\quad\quad\quad e^{-\phi|Y|^2/2}\mathrm{d}|Y|\mathrm{d}\hat{Y}+\mathrm{l.o.t.}.
\end{split}
\]
Here the above lower order term in $h$ can be ignored if we consider $N_\mathsf{F}$ as a semiclassical (scattering) pseudodifferential operator.

Notice that $\nu(\xi+\mathrm{i}\mathsf{F}))+2\alpha=\nu(\xi-\mathrm{i}\mathsf{F})=\alpha(\xi_\mathsf{F}-\mathrm{i})$, the above integral simplifies to
\[
\begin{split}
&\int_{\mathbb{S}^{n-2}}\int_0^\infty \sqrt{\nu}e^{\mathrm{i}|Y|\hat{Y}\cdot\eta}\left(
\begin{array}{cc}
1&-\mathrm{i}\nu(\xi-\mathrm{i}\mathsf{F})|Y|\langle \hat{Y},\,\cdot\,\rangle\\
-\mathrm{i}\nu(\xi+\mathrm{i}\mathsf{F})|Y|\hat{Y} & \mathrm{Id}+(-\nu^2(\xi^2+\mathsf{F}^2)|Y|^2-1)\hat{Y}\langle\hat{Y},\,\cdot\,\rangle
\end{array}
\right)\\
&\quad\quad\quad\quad \quad\quad  e^{-\phi|Y|^2/2}\mathrm{d}|Y|\mathrm{d}\hat{Y}+\mathrm{l.o.t.}.
\end{split}
\]

Extend the integral in $|Y|$ to $\mathbb{R}$, replacing it by a variable $t$ and use the fact that the integrand is invariant under the joint change of variables $t\rightarrow -t$ and $\hat{Y}\rightarrow -\hat{Y}$. This gives
\[
\begin{split}
\int_{\mathbb{S}^{n-2}}\int_\mathbb{R} \sqrt{\nu}e^{\mathrm{i}t\hat{Y}\cdot\eta}\left(
\begin{array}{cc}
1&-\mathrm{i}\nu(\xi-\mathrm{i}\mathsf{F})t\langle \hat{Y},\,\cdot\,\rangle\\
-\mathrm{i}\nu(\xi+\mathrm{i}\mathsf{F})t\hat{Y} & \mathrm{Id}+(-\nu^2(\xi^2+\mathsf{F}^2) t^2-1)\hat{Y}\langle\hat{Y},\,\cdot\,\rangle
\end{array}
\right)e^{-\phi t^2/2}\mathrm{d}t\mathrm{d}\hat{Y}
\end{split}
\]

Now the $t$ integral is a Fourier transform evaluated at $-\hat{Y}\cdot\eta$, under which multiplication by $t$ becomes $D_{\hat{Y}\cdot\eta}$. Note also that the Fourier transform of $e^{-\phi t^2/2}$ is a constant multiple of
\[
\phi^{-1/2}e^{-(\hat{Y}\cdot\eta)^2/(2\phi)}.
\]
We are left with
\begin{equation}\label{principalsymbolYintegral}
\begin{split}
\int_{\mathbb{S}^{n-2}} &\sqrt{\nu}\phi(\xi,\hat{Y})^{-1/2}\left(
\begin{array}{cc}
1&-\mathrm{i}\nu(\xi-\mathrm{i}\mathsf{F})D_{\hat{Y}\cdot\eta}\langle \hat{Y},\,\cdot\,\rangle\\
-\mathrm{i}\nu(\xi+\mathrm{i}\mathsf{F})D_{\hat{Y}\cdot\eta}\hat{Y} & \mathrm{Id}+(-\nu^2(\xi^2+\mathsf{F}^2) D_{\hat{Y}\cdot\eta}^2-1)\hat{Y}\langle\hat{Y},\,\cdot\,\rangle
\end{array}
\right)\\
&\quad\quad\quad\quad\quad\quad\quad\quad\quad e^{-(\hat{Y}\cdot\eta)^2/(2\phi(\xi,\hat{Y}))}\mathrm{d}\hat{Y}\\
=\int_{\mathbb{S}^{n-2}} &h^{-1}(\xi_\mathsf{F}^2+1)^{-1/2}\left(\begin{array}{cc}
1 & (\xi_\mathsf{F}-\mathrm{i})\frac{\hat{Y}\cdot\eta_\mathsf{F}}{\xi_\mathsf{F}^2+1}\langle\hat{Y},\,\cdot\,\rangle\\
(\xi_\mathsf{F}+\mathrm{i})\frac{\hat{Y}\cdot\eta_\mathsf{F}}{\xi^2_\mathsf{F}+1}\hat{Y} &\mathrm{Id}+(\frac{(\hat{Y}\cdot\eta_\mathsf{F})^2}{\xi_\mathsf{F}^2+1}-1) \hat{Y}\langle\hat{Y},\,\cdot\,\rangle
\end{array}
\right) \\
&\quad\quad\quad\quad\quad\quad\quad\quad\quad e^{-(\hat{Y}\cdot\eta_\mathsf{F})^2/(2h\phi_\mathsf{F}(\xi_\mathsf{F},\hat{Y}))}\mathrm{d}\hat{Y},
\end{split}
\end{equation}
where
\[
\phi_\mathsf{F}(\xi_\mathsf{F},\hat{Y})=\alpha(\xi_\mathsf{F}^2+1).
\]

If $v=(v_x,v_y)$ belongs to the kernel of the matrix

\begin{equation}\label{matrix_T1_fp}
\left(\begin{array}{cc}
1 & (\xi_\mathsf{F}-\mathrm{i})\frac{\hat{Y}\cdot\eta_\mathsf{F}}{\xi_\mathsf{F}^2+1}\langle\hat{Y},\,\cdot\,\rangle\\
(\xi_\mathsf{F}+\mathrm{i})\frac{\hat{Y}\cdot\eta_\mathsf{F}}{\xi^2_\mathsf{F}+1}\hat{Y} &\mathrm{Id}+(\frac{(\hat{Y}\cdot\eta_\mathsf{F})^2}{\xi_\mathsf{F}^2+1}-1) \hat{Y}\langle\hat{Y},\,\cdot\,\rangle
\end{array}
\right),
\end{equation}
 then
\begin{eqnarray}
v_x+(\xi_\mathsf{F}-\mathrm{i})\frac{\hat{Y}\cdot\eta_\mathsf{F}}{\xi_\mathsf{F}^2+1}\langle\hat{Y},v_y\rangle&=0,\label{eq_kernel1}\\
(\xi_\mathsf{F}+\mathrm{i})\frac{\hat{Y}\cdot\eta_\mathsf{F}}{\xi^2_\mathsf{F}+1}v_x\hat{Y}+ v_y+\left(\frac{(\hat{Y}\cdot\eta_\mathsf{F})^2}{\xi_\mathsf{F}^2+1}-1\right)\langle\hat{Y},v_y\rangle\hat{Y}&=0. \label{eq_kernel2}
\end{eqnarray}

Taking $\hat{Y}$ orthogonal to $\eta_\mathsf{F}$, i.e. $\hat{Y}\cdot\eta_\mathsf{F}=0$, then we immediately get $v_x=0$ from \eqref{eq_kernel1}, and \eqref{eq_kernel2} simplifies to
\[
v_y-\langle\hat{Y},v_y\rangle\hat{Y}=0,
\]
which means that $v_y\cdot\eta_\mathsf{F}=0$. If $\eta_\mathsf{F}=0$, then $\hat{Y}$ can be any unit-length vector. Then we can already conclude that $v_y=0$.

%
%
Now assume that $\eta_\mathsf{F}\neq 0$. Then taking $\hat{Y}=\epsilon \hat{\eta}_\mathsf{F}+\sqrt{1-\epsilon^2}\hat{\eta}_\mathsf{F}^\perp$, where $\epsilon$ is sufficiently close to $0$, $\hat{\eta}_\mathsf{F}^\perp$ is an arbitrary vector of unit length orthogonal to $\eta_\mathsf{F}$, in  \eqref{eq_kernel1}, we get
\[
\langle\epsilon \hat{\eta}_\mathsf{F}+\sqrt{1-\epsilon^2}\hat{\eta}_\mathsf{F}^\perp ,v_y\rangle=0.
\]
Taking second order derivative in $\epsilon$ at $\epsilon=0$, we get
\[
\langle \hat{\eta}_\mathsf{F}^\perp,v_y\rangle =0,
\]
and then
\[
v_x=0,\quad v_y=0.
\]
This proves that the integral of above matrix is positive definite. This shows that the principal symbol of $N_\mathsf{F}$ is elliptic at base infinity.
\end{proof}

The ellipticity proved by above two lemmas implies that there exists a local parametrix $P\in\Psi_{sc}^{1,0}(X)$ such that
\[
PN_\mathsf{F}=\mathrm{Id}+R
\]
with $R\in\Psi_{sc}^{0,0}(X)$, locally, smoothing, over some open subset of $X$ containing $\Omega_c$. Furthermore, $R$ is small for sufficiently small $c$ as an operator on some appropriate Sobolev spaces $H^{s,r}_{sc}(X)$ corresponding to the scattering structure. See \cite[Section 2]{UV} for more details.
 Therefore, we can show the following local invertibility result as in \cite{UV}.
\begin{theorem}
Assume $\partial M$ is strictly convex at $p\in\partial M$. There exists a function $\tilde{x}\in C^\infty(\widetilde{M})$ with $O_p=\{\tilde{x}>-c\}\cap M$ for sufficiently small $c>0$, such that a $1$-tensor $f$ can be uniquely determined by $T_1f$ restricted to $O_p$-local geodesics.
\end{theorem}
This local result lead immediately to the following global result using a layering stripping scheme \cite{UV}.
\begin{theorem}
Assume $\partial M$ is strictly convex and $(M,g)$ admits a smooth strictly convex function. If
\[
T_1f(\gamma,\eta)=0
\]
for any geodesic $\gamma$ in $M$ with endpoints on $\partial M$, and $\eta$ a parallel vector field conormal to $\gamma$. Then $f=0$.
\end{theorem}

\section{Mixed ray transform of $1+1$ tensors}\label{section4}


Recall the mixed ray transform of a symmetric $(1,1)$-tensor is defined by
\[
L_{1,1}f(x,y,\lambda,\omega)=\int_{\mathbb{R}}\mathcal{T}_{\gamma_{x,y,\lambda,\omega}}^{0,t}P_{\dot{\gamma}_{x,y,\lambda,\omega}(t),\vartheta_{x,y,\omega}(t)}\Lambda_{\dot{\gamma}_{x,y,\lambda,\omega}(t)}f^{sc}({\gamma_{x,y,\lambda,\omega}(t)})\mathrm{d}t,
\]
for $f\in\mathcal{B}S^1_1TX$. Here $\vartheta_{x,y,\omega}$ are chosen as in previous section.
Then we define
\[
Lv(x,y)=\int\left(\mathrm{Id}-\frac{\omega\partial_y\otimes g_{sc}(\lambda\partial_x+\omega \partial_y)}{h(\omega\partial_y,\omega\partial_y)}\right)\chi(\lambda/x)v(x,y,\lambda,\omega)\otimes g_{sc}(\lambda\partial_x+\omega\partial_y)\mathrm{d}\lambda\mathrm{d}\omega,
\]
and
\[
N_\mathsf{F}=e^{-\mathsf{F}/x}LL_{1,1}e^{\mathsf{F}/x}.
\]

Then one can see that
\[
N_\mathsf{F}\in \Psi_{sc}^{-1,0}(X;\mathcal{B}^{sc}S_1^1{^{sc}T}X,\mathcal{B}^{sc}S_1^1{^{sc}T}X).
\]
At the front face $x=0$,  the kernel of $N_\mathsf{F}$ is
\[
\begin{split}
&K^\flat(0,y,X,Y)\\
=&e^{-\mathsf{F}X}|Y|^{-n+1}\chi(S)\left(\mathrm{Id}-\left(\hat{Y}\cdot(x\partial_y)\right)\left(S\frac{\mathrm{d}x}{x^2}+\hat{Y}\cdot\frac{\mathrm{d}y}{x}\right)\right)\otimes \left(S\frac{\mathrm{d}x}{x^2}+\hat{Y}\cdot\frac{\mathrm{d}y}{x}\right)\mathcal{T}(0,y,X,Y)\\
&\left(\left(\mathrm{Id}-\left((S+2\alpha |Y|)(x^2\partial_x)+\hat{Y}\cdot(x\partial_y)\right)\left(\hat{Y}\cdot\frac{\mathrm{d}y}{x}\right)\right)\otimes\left((S+2\alpha |Y|)(x^2\partial_x)+\hat{Y}\cdot(x\partial_y)\right)\right).
\end{split}
\]
Using the basis
\[
(x^2\partial_x)\otimes\frac{\mathrm{d}x}{x^2},\,(x^2\partial_x)\otimes\frac{\mathrm{d}y}{x},\,(x\partial_y)\otimes\frac{\mathrm{d}x}{x^2},\,(x\partial_y)\otimes\frac{\mathrm{d}y}{x}
\]
for $(1,1)$-tensors, we can write $f$ in a vector-form
\[
f=\left(
\begin{array}{c}
f^x_x\\
f^x_y\\
f^y_x\\
f^y_y
\end{array}
\right),
\]
that is,
\[
f=f^x_x(x^2\partial_x)\otimes\frac{\mathrm{d}x}{x^2}+f^x_y(x^2\partial_x)\otimes\frac{\mathrm{d}y}{x}+f^y_x(x\partial_y)\otimes\frac{\mathrm{d}x}{x^2}+f^y_y(x\partial_y)\otimes\frac{\mathrm{d}y}{x}.
\]
The trace of $f$ is then
\[
\mathrm{trace}(f)=f_x^x+\langle\mathrm{Id},f^y_y\rangle:=f_x^x+\sum_{j=1}^{n-1}f^{y^j}_{y^j}.
\]
Since $f\in \mathcal{B}S^1_1TM$, we have $\mathrm{trace}(f)=0$.\\

The kernel of $N_\mathsf{F}$ at the scattering front face $x=0$ can be written as a matrix
\begin{equation}\label{matrix_L11_fiber}
e^{-\mathsf{F}X}\chi(S)|Y|^{-n+1}\left(
\begin{array}{cc}
S & 0\\
\hat{Y}_2 &0\\
0 & S\\
0 &\hat{Y}_2
\end{array}
\right)
\mathfrak{M}\left(
\begin{array}{cccc}
S+2\alpha|Y| &\langle\hat{Y},\,\cdot\,\rangle_2 & 0 &0\\
0&0& S+2\alpha|Y| &\langle\hat{Y},\,\cdot\,\rangle_2
\end{array}
\right),
\end{equation}
where
\[
\mathfrak{M}=\left(
\begin{array}{cc}
1 & -(S+2\alpha|Y|)\langle \hat{Y},\,\cdot\,\rangle_1\\
-S\hat{Y}_1 & I+(S(S+2\alpha|Y|)-1)\hat{Y}\langle\hat{Y},\,\cdot\,\rangle_1
\end{array}
\right),
\]
which is the same as the matrix already appeared in \eqref{schwartkernelT1}.
Here the subscripts $1$ and $2$ mean that they are acting on the first (upper/contravariant) and second (lower/covariant) factors respectively.\\

\subsection{Ellipticity}
Note that $\mathrm{d}=\nabla$ is the gradient with respect to the metric $g$.
Define
\[
\mathrm{d}^\mathcal{B}_\mathsf{F}=e^{-\mathsf{F}/x}\mathrm{d}^\mathcal{B}e^{\mathsf{F}/x}:C^\infty(S^1{^{sc}T}X)\rightarrow C^\infty(\mathcal{B}S^1_1{^{sc}T}X),
\]
and $-\delta$ be the formal adjoint of the operator $\mathrm{d}$ with respect to $g_{sc}$ (not $g$).
Denote
\[
\delta_\mathsf{F}^\mathcal{B}:C^\infty(\mathcal{B}S^1_1{^{sc}T}X)\rightarrow C^\infty(S^1{^{sc}T}X),\quad \delta_\mathsf{F}^\mathcal{B}=\delta_\mathsf{F}=e^{\mathsf{F}/x}\delta e^{-\mathsf{F}/x}.
\]

\begin{lemma}
The boundary symbol of $N_\mathsf{F}\in \Psi_{sc}^{-1,0}(X;\mathcal{B}S^1_1{^{sc}T}X,\mathcal{B}S^1_1{^{sc}T}X)$ is elliptic at fiber infinity of $^{sc}T^*X$ when restricted to the kernel of the principal symbol of $\delta^\mathcal{B}_\mathsf{F}$.
\end{lemma}
\begin{proof}
Similar to the proof of Lemma \ref{ellipticity10fiber},
we need to show that the integral of \eqref{matrix_L11_fiber}, restricted to the front face $|Y|=0$, after removing the singular factor $|Y|^{-n+1}$, along the equatorial sphere corresponding to any $\zeta=(\xi,\eta)\neq 0$, is positive definite on the kernel of the principal symbol of $\delta^\mathcal{B}_\mathsf{F}$ at fiber infinity. First, it is easy to see that it is positive semidefinite.

So we only need to consider the matrix
\begin{equation}\label{matrix_L11_ff}
\left(
\begin{array}{cc}
\tilde{S} & 0\\
\hat{Y}_2 &0\\
0 & \tilde{S}\\
0 &\hat{Y}_2
\end{array}
\right)\left(
\begin{array}{cc}
1 & -\tilde{S}\langle \hat{Y},\,\cdot\,\rangle_1\\
-\tilde{S}\hat{Y}_1 & I+(\tilde{S}^2-1)\hat{Y}_1\langle\hat{Y},\,\cdot\,\rangle_1
\end{array}
\right)\left(
\begin{array}{cccc}
\tilde{S} &\langle\hat{Y},\,\cdot\,\rangle_2 & 0 &0\\
0&0 & \tilde{S} &\langle\hat{Y},\,\cdot\,\rangle_2
\end{array}
\right).
\end{equation}
If $f$ belongs to the kernel of above matrix, we have
\begin{equation}\label{kernel11eq}
\begin{split}
\tilde{S}f^x_x+\langle\hat{Y},f^x_y\rangle-\tilde{S}^2\langle\hat{Y},f^y_x\rangle -\tilde{S}\langle\hat{Y}\otimes\hat{Y},f^y_y\rangle=0,\\
-\tilde{S}^2f^x_x\hat{Y}-\tilde{S}\langle\hat{Y},f^x_y+f^y_x\rangle\hat{Y}+\tilde{S}f^y_x+\langle\hat{Y},f^y_y\rangle_2+(\tilde{S}^2-1)\hat{Y}\langle\hat{Y}\otimes\hat{Y},f^y_y\rangle=0.
\end{split}
\end{equation}
If $f$ is in the kernel of the principal symbol of $\delta^\mathcal{B}_\mathsf{F}$, we have $\mathrm{trace}(f)=0$ and
\begin{equation}\label{deltakernel11}
\begin{split}
\xi f^x_x+\langle\eta,f^x_y\rangle=0,\\
\xi f^y_x+\langle\eta,f^y_y\rangle_2=0.
\end{split}
\end{equation}
Here and for the rest of the paper, we omit the subscript when there is no ambiguity.
Assume the scattering metric is in the form of $\frac{\mathrm{d}x^2}{x^4}+\frac{\mathrm{d}y^2}{x^2}$ at the point of interest.

Keep in mind that we need to take $(\tilde{S},\hat{Y})$ to be on the equatorial sphere $\xi\tilde{S}+\eta\cdot\hat{Y}=0$.
First, we take $\tilde{S}=0$ and $\hat{Y}=\hat{\eta}^\perp$, the first equation of \eqref{kernel11eq} gives
\begin{equation}\label{identityfxy1}
\langle\hat{\eta}^\perp,f^x_y\rangle=0,
\end{equation}
and the second equation of  \eqref{kernel11eq} gives
\begin{equation}\label{etamix222}
\langle\hat{\eta}^\perp,f^y_y\rangle_2=\langle\hat{\eta}^\perp\otimes \hat{\eta}^\perp,f^y_y\rangle\hat{\eta}^\perp,
\end{equation}
and consequently
\begin{equation}\label{etamix11}
\langle\eta\otimes\hat{\eta}^\perp,f^y_y\rangle=0.
\end{equation}

If $\eta=0$, $\hat{\eta}^\perp$ can be any unit-length vector, and we can conclude that $f^x_y=0$, and from \eqref{deltakernel11}
\[
f^x_x=0, \quad f^y_x=0.
\]
By \eqref{etamix222}, we have
\[
f^y_y=c\mathrm{Id},
\]
for some constant $c$.
With the trace-free condition $f^x_x+\langle\mathrm{Id},f^y_y\rangle=0$, we know that $c=0$ and thus $f^y_y=0$, and this completes the proof of $f=0$.\\

Now assume that $\eta\neq 0$. \\
\textit{Case 1}. $\xi\neq 0$: Take $\hat{Y}=\epsilon\hat{\eta}+\sqrt{1-\epsilon^2}\hat{\eta}^\perp$. Substitute $f^x_x=-\frac{1}{\xi}\langle\eta,f^x_y\rangle$ and $f^y_x=-\frac{1}{\xi}\langle\eta,f^y_y\rangle$ into the first equation of \eqref{kernel11eq}, we have
\[
\begin{split}
\epsilon\langle\left(\frac{|\eta|^2}{\xi^2}+1\right)\hat{\eta},f^x_y\rangle+\epsilon^3\langle\frac{|\eta|}{\xi}\left(\frac{|\eta|^2}{\xi^2}+1\right)\hat{\eta}\otimes\hat{\eta},f^y_y\rangle+\epsilon(1-\epsilon^2)\langle\frac{|\eta|}{\xi}\hat{\eta}^\perp\otimes\hat{\eta}^\perp,f^y_y\rangle\\
+\epsilon^2\sqrt{1-\epsilon^2}\langle \frac{|\eta|}{\xi}\left(\frac{|\eta|^2}{\xi^2}+1\right)\hat{\eta}^\perp\otimes\hat{\eta},f^y_y\rangle+\epsilon^2\sqrt{1-\epsilon^2}\langle \frac{|\eta|}{\xi}\hat{\eta}\otimes\hat{\eta}^\perp,f^y_y\rangle=0,
\end{split}
\]
(note that the last term vanishes). Taking 1st, 2nd and 3rd derivatives in $\epsilon$ at $\epsilon=0$, we obtain
\begin{eqnarray}
\langle\left(\frac{|\eta|^2}{\xi^2}+1\right)\hat{\eta},f^x_y\rangle+\langle\frac{|\eta|}{\xi}\hat{\eta}^\perp\otimes \hat{\eta}^\perp,f^y_y\rangle=0,\label{1stderivative}\\
\langle\hat{\eta}^\perp\otimes\hat{\eta},f^y_y\rangle=0,\label{2ndderivative}\\
\langle\left(\frac{|\eta|^2}{\xi^2}+1\right)\hat{\eta}\otimes\hat{\eta},f^y_y\rangle-\langle\hat{\eta}^\perp\otimes \hat{\eta}^\perp,f^y_y\rangle=0.\label{3rderivative}
\end{eqnarray}
Inserting the first equation of \eqref{deltakernel11} into \eqref{1stderivative}, we have
\begin{equation}\label{first_substi}
-\frac{\xi}{|\eta|}\left(\frac{|\eta|^2}{\xi^2}+1\right)f^x_x+\frac{|\eta|}{\xi}\langle\hat{\eta}^\perp\otimes \hat{\eta}^\perp,f^y_y\rangle=0
\end{equation}
Since $f$ is trace-free, then
we can take $\upsilon_j$, $j=1,\cdots,n-2$ such that $(\hat{\eta}_\mathsf,\upsilon_1,\cdots,\upsilon_{n-2})$ is an orthonormal basis for $\mathbb{R}^{n-1}$
and
\begin{equation}\label{tracekernel11}
f^x_x+\langle \hat{\eta}\otimes\hat{\eta},f^y_y\rangle+\sum_{j=1}^{n-2}\langle \upsilon_j\otimes\upsilon_j,f^y_y\rangle=0.
\end{equation}
We can then summarize \eqref{3rderivative} \eqref{first_substi} \eqref{tracekernel11} into the following linear system
\[
\left(\begin{array}{ccc}
1 & 1 &  \mathbbm{1}^T\\
-\left(\frac{|\eta|^2}{\xi^2}+1\right) & 0 & \frac{|\eta|^2}{\xi^2}\\
0 & \left(\frac{|\eta|^2}{\xi^2}+1\right) \mathbbm{1} & -\mathrm{Id}
\end{array}
\right)\left(\begin{array}{c}
f^x_x\\
\langle\hat{\eta}\otimes \hat{\eta},f^y_y\rangle\\
\langle\upsilon_j\otimes \upsilon_j,f^y_y\rangle_{j=1}^{n-2}
\end{array}
\right)=0,
\]
where 
\[
\mathbbm{1}=(\underbrace{1,\cdots,1}_{n-2})^T.
\]
The above linear system is nonsingular, thus
\[
f^x_x=\langle\hat{\eta}\otimes \hat{\eta},f^y_y\rangle=\langle\hat{\eta}^\perp\otimes \hat{\eta}^\perp,f^y_y\rangle=0,\quad \text{for any }\hat{\eta}^\perp\perp\hat{\eta}.
\]
Recall that we already proved $\langle\hat{\eta}\otimes\hat{\eta}^\perp,f^y_y\rangle=\langle\hat{\eta}^\perp\otimes\hat{\eta},f^y_y\rangle=0$ (cf. \eqref{etamix11}\eqref{2ndderivative}), we end up with
\[
f=0.
\]
\textit{Case 2}. $\xi= 0$: By \eqref{deltakernel11}, we have
\[
\begin{split}
\langle \hat{\eta},f^x_y\rangle =0,\\
\langle \hat{\eta},f^y_y\rangle_2=0.
\end{split}
\]
In particular
\[
\begin{split}
\langle \hat{\eta}\otimes \hat{\eta},f^y_y\rangle=0,\\
\langle \hat{\eta}^\perp\otimes \hat{\eta},f^y_y\rangle=0,
\end{split}
\]
and (taking \eqref{identityfxy1} into account)
\[
f^x_y=0.
\]
Now take any $\tilde{S}\neq 0$ sufficiently small and $\hat{Y}=\hat{\eta}^\perp$ in the first equation of \eqref{kernel11eq}, we have
\[
f^x_x-\tilde{S}\langle\hat{\eta}^\perp,f^y_x\rangle -\langle\hat{\eta}^\perp\otimes\hat{\eta}^\perp,f^y_y\rangle=0.
\]
Varying the value of $\tilde{S}$ we have
\[
f^x_x -\langle\hat{\eta}^\perp\otimes\hat{\eta}^\perp,f^y_y\rangle=0
\]
and
\[
\langle\hat{\eta}^\perp,f^y_x\rangle =0.
\]
for any $\hat{\eta}^\perp\perp\hat{\eta}$.
Together with \eqref{tracekernel11}
\[
f^x_x+\sum_{j=1}^{n-2}\langle\upsilon_j\otimes \upsilon_j,f^y_y\rangle=0,
\]
we have
\[
f^x_x=\langle \hat{\eta}^\perp\otimes\hat{\eta}^\perp,f^y_y\rangle=0,
\]
for any $\hat{\eta}^\perp\perp\hat{\eta}$.
Recall \eqref{etamix11}, we have
\[
f^y_y=0.
\]
Finally, we take $\hat{Y}=\hat{\eta}^\perp$ in the second equation of \eqref{kernel11eq} to get
\[
f^y_x=0.
\]

The above argument shows that if $f$ is in the kernel of the principal symbol of $\delta_\mathsf{F}$ and in the kernel of the matrix \eqref{matrix_L11_ff} for any $(\tilde{S},\hat{Y})$ on the equatorial sphere $\xi\tilde{S}+\eta\cdot\hat{Y}$, we must have $f\equiv 0$.
Notice that in above we can always take $\tilde{S}$ sufficiently close to $0$, such that $\chi(\tilde{S})>0$, so the integral of
\[
\chi(\tilde{S})\left(
\begin{array}{cc}
\tilde{S} & 0\\
\hat{Y}_2 &0\\
0 & \tilde{S}\\
0 &\hat{Y}_2
\end{array}
\right)\left(
\begin{array}{cc}
1 & -\tilde{S}\langle \hat{Y},\,\cdot\,\rangle_1\\
-\tilde{S}\hat{Y}_1 & I+(\tilde{S}^2-1)\hat{Y}_1\langle\hat{Y},\,\cdot\,\rangle_1
\end{array}
\right)\left(
\begin{array}{cccc}
\tilde{S} &\langle\hat{Y},\,\cdot\,\rangle_2 & 0 &0\\
0&0 & \tilde{S} &\langle\hat{Y},\,\cdot\,\rangle_2
\end{array}
\right)
\]
over the equatorial sphere is positive definite on tensors in the kernel of the principal symbol of $\delta^{\mathcal{B}}_\mathsf{F}$.
\end{proof}

\begin{lemma}
The operator $N_\mathsf{F}\in \Psi_{sc}^{-1,0}(X;\mathcal{B}S^1_1{^{sc}T}X,\mathcal{B}S^1_1{^{sc}T}X)$ is elliptic at finite points of $^{sc}T^*X$ when restricted to the kernel of the principal symbol of $\delta^\mathcal{B}_\mathsf{F}$ for $\mathsf{F}>0$ large enough.
\end{lemma}

\begin{proof}
If a trace-free $(1,1)$-tensor $f$ is in the kernel of the principal symbol of $\delta^\mathcal{B}_\mathsf{F}$ as a semiclassical pseudodifferential operator, we have
\begin{equation}\label{deltakernel11sc}
\begin{split}
(\xi_\mathsf{F}-\mathrm{i}) f^x_x+\langle\eta_\mathsf{F},f^x_y\rangle_2=0,\\
(\xi_\mathsf{F}-\mathrm{i}) f^y_x+\langle\eta_\mathsf{F},f^y_y\rangle_2=0.
\end{split}
\end{equation}

Analogous to \eqref{principalsymbolYintegral} for transverse ray transform, the (semiclassical) principal symbol of $N_\mathsf{F}$ at the scattering front face is
\[
\int_{\mathbb{S}^{n-2}} (\xi_\mathsf{F}^2+1)^{-1/2}\mathcal{M}e^{-(\hat{Y}\cdot\eta_\mathsf{F})^2/(2h\phi_\mathsf{F}(\xi_\mathsf{F},\hat{Y}))}\mathrm{d}\hat{Y},
\]
where
\[
\begin{split}
\mathcal{M}=\left(
\begin{array}{cc}
-(\xi_\mathsf{F}+\mathrm{i})\left(\frac{\hat{Y}\cdot\eta_\mathsf{F}}{\xi_\mathsf{F}^2+1}\right)& 0\\
\hat{Y}_2&0\\
0&-(\xi_\mathsf{F}+\mathrm{i})\left(\frac{\hat{Y}\cdot\eta_\mathsf{F}}{\xi_\mathsf{F}^2+1}\right)\\
0&\hat{Y}_2
\end{array}
\right)
\left(\begin{array}{cc}
1 & (\xi_\mathsf{F}-\mathrm{i})\frac{\hat{Y}\cdot\eta_\mathsf{F}}{\xi_\mathsf{F}^2+1}\langle\hat{Y},\,\cdot\,\rangle_1\\
(\xi_\mathsf{F}+\mathrm{i})\frac{\hat{Y}\cdot\eta_\mathsf{F}}{\xi^2_\mathsf{F}+1}\hat{Y}_1 &\mathrm{Id}+(\frac{(\hat{Y}\cdot\eta_\mathsf{F})^2}{\xi_\mathsf{F}^2+1}-1) \hat{Y}_1\langle\hat{Y},\,\cdot\,\rangle_1
\end{array}
\right)\\
\left(
\begin{array}{cccc}
-(\xi_\mathsf{F}-\mathrm{i})\left(\frac{\hat{Y}\cdot\eta_\mathsf{F}}{\xi_\mathsf{F}^2+1}\right) &\langle\hat{Y},\,\cdot\,\rangle_2 & 0 &0\\
0&0&-(\xi_\mathsf{F}-\mathrm{i})\left(\frac{\hat{Y}\cdot\eta_\mathsf{F}}{\xi_\mathsf{F}^2+1}\right) &\langle\hat{Y},\,\cdot\,\rangle_2 
\end{array}
\right).
\end{split}
\]
See also the calculation in the proofs of \cite[Lemma 3.5]{stefanov2018inverting} and \cite[Lemma 2.4]{de2018mixed}.

If $f$ is in the kernel of above matrix, we have
\begin{equation}\label{identity11sc}
\begin{split}
-(\xi_\mathsf{F}-\mathrm{i})\left(\frac{\hat{Y}\cdot\eta_\mathsf{F}}{\xi_\mathsf{F}^2+1}\right)f^x_x+\langle\hat{Y},f^x_y\rangle-(\xi_\mathsf{F}-\mathrm{i})^2\left(\frac{\hat{Y}\cdot\eta_\mathsf{F}}{\xi_\mathsf{F}^2+1}\right)^2\langle\hat{Y},f^y_x\rangle&\\
+(\xi_\mathsf{F}-\mathrm{i})\left(\frac{\hat{Y}\cdot\eta_\mathsf{F}}{\xi_\mathsf{F}^2+1}\right)\langle\hat{Y}\otimes\hat{Y},f^y_y\rangle&=0,
\end{split}
\end{equation}
and by simply taking $\hat{Y}\cdot\eta_\mathsf{F}=0$,
\begin{equation}\label{eq_L11_1}
\langle\hat{Y},f^y_y\rangle_2-\langle\hat{Y}\otimes\hat{Y},f^y_y\rangle\hat{Y}=0,\text{ for }\hat{Y}\perp\eta_\mathsf{F}.
\end{equation}
Taking $\hat{Y}\cdot\eta_\mathsf{F}=0$ in \eqref{identity11sc} yields
\begin{equation}\label{eta0fxy}
\langle\hat{Y},f^x_y\rangle=0,\text{ for }\hat{Y}\perp\eta_\mathsf{F}.
\end{equation}

If $\eta_\mathsf{F}=0$,  $\hat{Y}$ can any unit-length vector, so $f^x_y=0$. And we could directly get $f^x_x=f^y_x=0$ from \eqref{deltakernel11sc}. Furthermore, \eqref{eq_L11_1} becomes
\[
\langle\hat{Y},f^y_y\rangle_2-\langle\hat{Y}\otimes\hat{Y},f^y_y\rangle\hat{Y}=0,\text{ for any }\hat{Y},
\]
which means that any vector is an eigenvector of $f^y_y$.
Therefore $f^y_y=c\mathrm{Id}_{n-1}$ with a constant $c$. Now the trace-free condition $f^x_x+\langle\mathrm{Id},f^y_y\rangle=0$ gives $c=0$ and consequently $f^y_y=0$. Then we already have $f=0$.\\

So in the following we assume that $\eta_\mathsf{F}\neq 0$.
Choose $\hat{\eta}_\mathsf{F}^\perp\perp\hat{\eta}_\mathsf{F}$. Taking inner product of the left hand side of \eqref{eq_L11_1} with $\hat{\eta}_\mathsf{F}$, we have
\begin{equation}\label{eta11eq}
\langle\hat{\eta}_\mathsf{F}\otimes\hat{\eta}_\mathsf{F}^\perp,f^y_y\rangle=0.
\end{equation}
Substituting \eqref{deltakernel11sc} into \eqref{identity11sc} yields
\[
-(\xi_\mathsf{F}-\mathrm{i})\left(\frac{\hat{Y}\cdot\eta_\mathsf{F}}{\xi_\mathsf{F}^2+1}\right)f^x_x+\langle\hat{Y},f^x_y\rangle+(\xi_\mathsf{F}-\mathrm{i})\left(\frac{\hat{Y}\cdot\eta_\mathsf{F}}{\xi_\mathsf{F}^2+1}\right)^2\langle\hat{Y}\otimes\eta_\mathsf{F},f^y_y\rangle+(\xi_\mathsf{F}-\mathrm{i})\left(\frac{\hat{Y}\cdot\eta_\mathsf{F}}{\xi_\mathsf{F}^2+1}\right)\langle\hat{Y}\otimes\hat{Y},f^y_y\rangle=0.
\]
Taking $\hat{Y}=\epsilon\hat{\eta}_\mathsf{F}+\sqrt{1-\epsilon^2}\hat{\eta}_\mathsf{F}^\perp$ in above identity, we have
\[
\begin{split}
-\epsilon(\xi_\mathsf{F}-\mathrm{i})\left(\frac{|\eta_\mathsf{F}|}{\xi_\mathsf{F}^2+1}+\frac{1}{|\eta_\mathsf{F}|}\right)f^x_x+\sqrt{1-\epsilon^2}\langle\hat{\eta}^\perp_\mathsf{F},f^x_y\rangle+\epsilon^3(\xi_\mathsf{F}-\mathrm{i})\langle\left(\frac{|\eta_\mathsf{F}|^3}{(\xi_\mathsf{F}^2+1)^2}+\frac{|\eta_\mathsf{F}|}{\xi_\mathsf{F}^2+1}\right)\hat{\eta}_\mathsf{F}\otimes \hat{\eta}_\mathsf{F},f^y_y\rangle\\
+\epsilon^2\sqrt{1-\epsilon^2}(\xi_\mathsf{F}-\mathrm{i})\langle\left(\frac{|\eta_\mathsf{F}|^3}{(\xi_\mathsf{F}^2+1)^2}+\frac{|\eta_\mathsf{F}|}{\xi_\mathsf{F}^2+1}\right)\hat{\eta}_\mathsf{F}^\perp\otimes\hat{\eta}_\mathsf{F},f^y_y\rangle+\epsilon^2\sqrt{1-\epsilon^2}(\xi_\mathsf{F}-\mathrm{i})\langle\frac{|\eta_\mathsf{F}|}{\xi_\mathsf{F}^2+1}\hat{\eta}_\mathsf{F}\otimes\hat{\eta}^\perp_\mathsf{F},f^y_y\rangle\\
+\epsilon(1-\epsilon^2)(\xi_\mathsf{F}-\mathrm{i})\langle\frac{|\eta_\mathsf{F}|}{\xi_\mathsf{F}^2+1}\hat{\eta}^\perp_\mathsf{F}\otimes\hat{\eta}^\perp_\mathsf{F},f^y_y\rangle=0.
\end{split}
\]

 We take 1st order derivative to obtain
\[
-\left(\frac{|\eta_\mathsf{F}|}{\xi_\mathsf{F}^2+1}+\frac{1}{|\eta_\mathsf{F}|}\right)f^x_x+\langle\frac{|\eta_\mathsf{F}|}{\xi_\mathsf{F}^2+1}\hat{\eta}^\perp_\mathsf{F}\otimes\hat{\eta}^\perp_\mathsf{F},f^y_y\rangle=0.
\]
Taking 2nd derivative gives
\begin{equation}\label{eta01eq}
\langle\hat{\eta}^\perp_\mathsf{F}\otimes \hat{\eta}_\mathsf{F},f^y_y\rangle=0.
\end{equation}
Taking 3rd order derivative of above equation in $\epsilon$, we have
\[
\langle\left(\frac{|\eta_\mathsf{F}|^3}{(\xi_\mathsf{F}^2+1)^2}+\frac{|\eta_\mathsf{F}|}{\xi_\mathsf{F}^2+1}\right)\hat{\eta}_\mathsf{F}\otimes \hat{\eta}_\mathsf{F},f^y_y\rangle-\langle\frac{|\eta_\mathsf{F}|}{\xi_\mathsf{F}^2+1}\hat{\eta}^\perp_\mathsf{F}\otimes\hat{\eta}^\perp_\mathsf{F},f^y_y\rangle=0.
\]
Take $\upsilon_j$, $j=1,\cdots,n-2$ such that $(\hat{\eta}_\mathsf,\upsilon_1,\cdots,\upsilon_{n-2})$ is an orthonormal basis for $\mathbb{R}^{n-1}$.
We then write
\[
\left(\begin{array}{ccc}
1 & 1 & \mathbbm{1}^T\\
-\frac{|\eta_\mathsf{F}|^2}{\xi_\mathsf{F}^2+1}-1 &0 & \frac{|\eta_\mathsf{F}|^2}{\xi_\mathsf{F}^2+1}\\
0 & \left(\frac{|\eta_\mathsf{F}|^2}{\xi_\mathsf{F}^2+1}+1\right)\mathbbm{1}&- \mathrm{Id}
\end{array}
\right)\left(\begin{array}{c}
f^x_x\\
\langle\hat{\eta}_\mathsf{F}\otimes \hat{\eta}_\mathsf{F},f^y_y\rangle\\
\langle\upsilon_j\otimes \upsilon_j,f^y_y\rangle_{j=1}^{n-2}
\end{array}
\right)=0,
\]
Since the above matrix is nonsingular, we can conclude that
\[
f^x_x=\langle \hat{\eta}_\mathsf{F}\otimes\hat{\eta}_\mathsf{F},f^y_y\rangle=\langle \hat{\eta}_\mathsf{F}^\perp\otimes\hat{\eta}_\mathsf{F}^\perp,f^y_y\rangle=0, \text{ for any }\eta^\perp_\mathsf{F}\perp\eta_\mathsf{F}.
\]
With \eqref{eta11eq} and \eqref{eta01eq}, we also have $f^y_y=0$. Using the second equation of \eqref{deltakernel11sc} yields
$f^y_x=0$. Using the first equation of \eqref{identity11sc} again, we see that 
\[
\langle \hat{\eta}_\mathsf{F},f^x_y\rangle=0
\]
together with \eqref{eta0fxy}, we get $f^x_y=0$. So finally we arrived at $f\equiv 0$ and this proves the ellipticity.
\end{proof}
\subsection{The gauge condition and the proof of the main results}
In the following we compute the principal symbols of $\mathrm{d}^\mathcal{B}_\mathsf{F}$ and  $\delta^\mathcal{B}_\mathsf{F}$. Recall that these operators are defined using the background metric $g$ and the adjoint $\delta_\mathsf{F}^\mathcal{B}$ is taken with respect to the scattering metric $g_{sc}$. We will use basis 
\[
x^2\partial_x,\,x\partial_y
\]
 for $(1,0)$-tensors, and
\[
(x^2\partial_x)\otimes\frac{\mathrm{d}x}{x^2},\,(x^2\partial_x)\otimes\frac{\mathrm{d}y}{x},\,(x\partial_y)\otimes\frac{\mathrm{d}x}{x^2},\,(x\partial_y)\otimes\frac{\mathrm{d}y}{x}
\]
as a basis for $(1,1)$-tensors.

Assume
\begin{equation}\label{form_f}
v=v^xx^2\partial_x+v^{y^i}x\partial_{y^i}.
\end{equation}
We calculate
\[
\begin{split}
(\nabla v)^x_x&=x^{-2}\partial_xv^x+\mathcal{O}(x^{-3}),\\
(\nabla v)^x_{y^i}&=x^{-2}\partial_{y^i}v^x+\mathcal{O}(x^{-2}),\\
(\nabla v)^{y^i}_{x}&=x^{-1}\partial_xv^{y^i}+\mathcal{O}(x^{-2}),\\
(\nabla v)^{y^i}_{y^j}&=x^{-1}\partial_{y^j}v^{y^i}+x^{-2}a(v_x)+\mathcal{O}(x^{-1}),
\end{split}
\]
where $a$ is related to the Christoffel symbols of the background metric $g$ under the local coordinates $(x,y)$.
Then one can write
\[
\begin{split}
\mathrm{d}v=&x^2\partial_xv^x(x^2\partial_x)\otimes \frac{\mathrm{d}x}{x^2}+x\partial_{y^i}v^x(x^2\partial_x)\otimes \frac{\mathrm{d}y^i}{x}+x^2\partial_{x}v^{y^i} (x\partial_{y^i})\otimes\frac{\mathrm{d}x}{x^2}\\
&\quad\quad\quad\quad+\left(x\partial_{y^j}v^{y^i}+a_{ij}(v_x)\right)(x\partial_{y^i})\otimes\frac{\mathrm{d}y^j}{x}+\mathrm{l.o.t.}.
\end{split}
\]
Therefore the principal symbol of $\frac{1}{\mathrm{i}}\mathrm{d}_\mathsf{F}$ is, in matrix form,
\[
\left(\begin{array}{cc}
\xi+\mathrm{i}\mathsf{F} &0\\
\eta_2 & 0\\
0&\xi+\mathrm{i}\mathsf{F} \\
a &\eta_2
\end{array}\right).
\]

Since we can perform the computation of the symbol at a single point, we assume that $g_{sc}$ is of the form $\frac{\mathrm{d}x^2}{x^4}+\frac{\mathrm{d}y^2}{x^2}$. We denote
\[
\mathrm{Id}=\mathrm{diag}\{\underbrace{1,1,\cdots,1}_{n-1}\}.
\]
Note that for $(1,1)$-tensor $f$, the operator $\mathcal{B}$ acts on $f$ as
\[
(\mathcal{B}f)^i_j=f^i_j-\frac{1}{n}\mathrm{trace}(f)\delta^i_j.
\]
The principal symbol of $\frac{1}{\mathrm{i}}\mathrm{d}^\mathcal{B}_\mathsf{F}=\frac{1}{\mathrm{i}}(\mathrm{d}_\mathsf{F}-\mathcal{B}\mathrm{d}_\mathsf{F})$ is then
\[
\left(\begin{array}{cc}
\xi+\mathrm{i}\mathsf{F} &0\\
\eta & 0\\
0&\xi+\mathrm{i}\mathsf{F} \\
a &\eta_2
\end{array}\right)-\frac{1}{n}\left(\begin{array}{c}
1\\
0\\
0\\
\mathrm{Id}
\end{array}
\right)
(1,0,0,\langle\mathrm{Id},\,\cdot\,\rangle)\left(\begin{array}{cc}
\xi+\mathrm{i}\mathsf{F} &0\\
\eta & 0\\
0&\xi+\mathrm{i}\mathsf{F} \\
a &\eta_2
\end{array}\right),
\]
and the principal symbol of $\frac{1}{\mathrm{i}}\delta^\mathcal{B}_\mathsf{F}$ is
\[
\left(\begin{array}{cccc}
\xi-\mathrm{i}\mathsf{F} &\langle \eta,\,\cdot\,\rangle & 0 &\langle a,\,\cdot\,\rangle\\
0 & 0 &\xi-\mathrm{i}\mathsf{F} &\langle \eta,\,\cdot\,\rangle_2
\end{array}\right)-
\frac{1}{n}\left(\begin{array}{cccc}
\xi-\mathrm{i}\mathsf{F} & \langle \eta,\,\cdot\,\rangle  & 0 & \langle a,\,\cdot\,\rangle\\
0 & 0 &\xi-\mathrm{i}\mathsf{F} &\langle \eta,\,\cdot\,\rangle_2
\end{array}\right)\left(\begin{array}{c}
1\\
0\\
0\\
\mathrm{Id}
\end{array}
\right)
(1,0,0,\langle\mathrm{Id},\,\cdot\,\rangle).
\]
Here 
\[
\langle \eta,w\rangle_2=\sum_j^{n-1}\eta_jw_{ij}.
\]

Introducing the operator
\[
\Delta^\mathcal{B}_\mathsf{F}=\delta^\mathcal{B}_\mathsf{F}d^\mathcal{B}_\mathsf{F},
\]
we have
\begin{lemma}
For $\mathsf{F}>0$ sufficiently large, $-\Delta^\mathcal{B}_\mathsf{F}$ is elliptic in $\mathrm{Diff}_{sc}^{2,0}(X;S^1{^{sc}TX},S^1{^{sc}TX})$ on $(1,0)$-tensors.
\end{lemma}
\begin{proof}
By calculation the principal symbol of $-\Delta^\mathcal{B}_\mathsf{F}$ is 
\[
\left(\begin{array}{cc}
\xi^2+\mathsf{F}^2+|\eta|^2 & 0\\
0&\xi^2+\mathsf{F}^2+|\eta|^2
\end{array}
\right)-\frac{1}{n}
\left(\begin{array}{cc}
\xi^2+\mathsf{F}^2 &(\xi-\mathrm{i}\mathsf{F})\langle\eta,\,\cdot\,\rangle\\
(\xi+\mathrm{i}\mathsf{F})\eta &\eta\langle\eta,\,\cdot\,\rangle
\end{array}
\right)+R(a),
\]
where $R(a)$ is the principal symbol of an operator in $\mathrm{Diff}^1_{sc}(X,S^1{^{sc}TX};S^1{^{sc}TX})+\mathsf{F}\mathrm{Diff}^0_{sc}(X,S^1{^{sc}TX};S^1{^{sc}TX})$.
Denote $\nabla_\mathsf{F}=e^{-\mathsf{F}/x}\nabla e^{-\mathsf{F}/x}$ with $\nabla$ gradient relative to $g_{sc}$ (not $g$) and $\nabla_\mathsf{F}^*$ the adjoint of $\nabla_\mathsf{F}$.
By the discussion in \cite{stefanov2018inverting}, $\nabla^*_\mathsf{F}\nabla_\mathsf{F}$ has the principal symbol
\[
\left(
\begin{array}{cc}
\xi^2+\mathsf{F}^2+|\eta|^2 & 0\\
0 &\xi^2+\mathsf{F}^2+|\eta|^2
\end{array}
\right).
\]
Note that
\[
\begin{split}
&M(\xi,\eta,\mathsf{F})\\
=&2\left(\begin{array}{cc}
\xi^2+\mathsf{F}^2 +|\eta|^2& 0\\
0&\xi^2+\mathsf{F}^2 +|\eta|^2
\end{array}
\right)-\left(\begin{array}{cc}
\xi^2+\mathsf{F}^2 &(\xi-\mathrm{i}\mathsf{F})\langle\eta,\,\cdot\,\rangle\\
(\xi+\mathrm{i}\mathsf{F})\eta &\eta\langle\eta,\,\cdot\,\rangle
\end{array}
\right)\\
=&2\left(\begin{array}{cc}
|\eta|^2 & 0\\
0&\xi^2+\mathsf{F}^2
\end{array}
\right)+\left(\begin{array}{cc}
\xi^2+\mathsf{F}^2 &-(\xi-\mathrm{i}\mathsf{F})\langle\eta,\,\cdot\,\rangle\\
-(\xi+\mathrm{i}\mathsf{F})\eta &2|\eta|^2-\eta\langle\eta,\,\cdot\,\rangle
\end{array}
\right)
\end{split}
\]
is positive semidefinite. To see this, one can verify
\[
\begin{split}
&(\overline{f_x},\overline{f_y})
\left(\begin{array}{cc}
\xi^2+\mathsf{F}^2 &-(\xi-\mathrm{i}\mathsf{F})\langle\eta,\,\cdot\,\rangle\\
-(\xi+\mathrm{i}\mathsf{F})\eta &2|\eta|^2-\eta\langle\eta,\,\cdot\,\rangle
\end{array}
\right)\left(\begin{array}{c}
f_x\\
f_y
\end{array}
\right)\\
=&(\xi^2+\mathsf{F}^2)|f_x|^2-(\xi-\mathrm{i}\mathsf{F})\overline{f_x}\eta_jf_{y_j}-(\xi+\mathrm{i}\mathsf{F})f_x\eta_j\overline{f_{y_j}}+2|\eta|^2|f_y|^2-(\eta_jf_{y_j})(\eta_i\overline{f_{y_i}})\\
\geq &(\xi^2+\mathsf{F}^2)|f_x|^2-|\xi-\mathrm{i}\mathsf{F}|^2|f_x|^2-2|\eta_jf_{y_j}|^2+2|\eta|^2|f_y|^2\\
\geq &0.
\end{split}
\]
Then
\[
\sigma(-\Delta^\mathcal{B}_\mathsf{F})=\frac{n-2}{n}\sigma(\nabla^*_\mathsf{F}\nabla_\mathsf{F})+\frac{1}{n}M(\xi,\eta,\mathsf{F})+R(a).
\]
It is easy to see that 
\[
\frac{n-2}{n}\sigma(\nabla^*_\mathsf{F}\nabla_\mathsf{F})+\frac{1}{n}M(\xi,\eta,\mathsf{F})\geq C(\mathsf{F}^2+|\xi|^2+|\eta|^2)\mathrm{Id},
\]
and
\[
|R(a)|\leq \epsilon^2(|\xi|^2+|\eta|^2)+\frac{C}{\epsilon^2}.
\]
Then we can see that if $\mathsf{F}$ is sufficiently large, $-\Delta^\mathcal{B}_\mathsf{F}$ is elliptic.\\
\end{proof}

Now we write
\[
\begin{split}
&M(\xi,\eta,\mathsf{F})\\
=&2\left(\begin{array}{cc}
|\eta|^2 & 0\\
0&\xi^2+\mathsf{F}^2
\end{array}
\right)+\left(\begin{array}{cc}
\xi^2+\mathsf{F}^2 &-(\xi-\mathrm{i}\mathsf{F})\langle\eta,\,\cdot\,\rangle\\
-(\xi+\mathrm{i}\mathsf{F})\eta &\eta\langle\eta,\,\cdot\,\rangle
\end{array}
\right)+2\left(\begin{array}{cc}
0 &0\\
0 &|\eta|^2-\eta\langle\eta,\,\cdot\,\rangle
\end{array}
\right).
\end{split}
\]
Now, denote $\nabla_y$ to be the gradient operator on $Y$ with respect to the Riemannian metric $\frac{h(x,y)\mathrm{d}y^2}{x^2}$ and $\nabla^*_y$ the adjoint of $\nabla_y$ with respect to $\frac{h(x,y)\mathrm{d}y^2}{x^2}$. 

We introduce the following operators. We remark here that all the adjoints are taken with respect to the scattering metric $g_{sc}=\frac{\mathrm{d}x^2}{x^4}+\frac{h(x,y)\mathrm{d}y^2}{x^2}$.
With $w=w_x\partial_x+w_y\partial_y\in S^1{^{sc}TX}$ written in vector form $w=\left(\begin{array}{c}w_x\\w_y\end{array}\right)$, we denote
\[
\mathfrak{D}_1=\left(\begin{array}{cc}
0&0\\
\nabla_y &0\\
0 & 0\\
0&0
\end{array}
\right):S^1{^{sc}TX}\rightarrow S^1_1{^{sc}TX},\quad\quad \mathfrak{D}_2=\left(\begin{array}{cc}
0&0\\
0 &0\\
0 & \partial_x+\mathsf{F}\\
0&0
\end{array}
\right):S^1{^{sc}TX}\rightarrow S^1_1{^{sc}TX}.
\]
Then
\[
\mathfrak{D}_1f=x\partial_{y^i}f_x(x^2\partial_x)\otimes\frac{\mathrm{d}y^i}{x}+\mathrm{l.o.t.},\quad \mathfrak{D}_2f=x^2\partial_xf_{y^i}(x\partial_{y^i})\otimes\frac{\mathrm{d}x}{x^2}+\mathrm{l.o.t.},
\]
where $f$ is of the form \eqref{form_f}.
One can verify that
\[
\sigma
\left(\mathfrak{D}_1^*\mathfrak{D}_1+\mathfrak{D}_2^*\mathfrak{D}_2\right)=\left(\begin{array}{cc}
|\eta|^2 & 0\\
0&\xi^2+\mathsf{F}^2
\end{array}
\right).
\]
Let
\[
\mathfrak{D}_3=\left(\begin{array}{c}
\partial_x+\mathsf{F}\\
-\nabla_y
\end{array}
\right):C^\infty(X)\rightarrow  S^1{^{sc}TX}.
\]
Then
\[
\mathfrak{D}^*_3f=-\partial_xf_x+\mathsf{F}f_x+\partial_{y^i}f_{y^i}+\mathrm{l.o.t.}.
\]
Then one can calculate
\[
\sigma\left(\mathfrak{D}_3\mathfrak{D}_3^*
\right)=\left(\begin{array}{cc}
\xi^2+\mathsf{F}^2 &-(\xi-\mathrm{i}\mathsf{F})\langle\eta,\,\cdot\,\rangle\\
-(\xi+\mathrm{i}\mathsf{F})\eta &\eta\langle\eta,\,\cdot\,\rangle
\end{array}
\right).
\]
Now let
\[
\mathfrak{D}_4=\left(\begin{array}{cc}
0&0\\
0&0\\
0&0\\
0&\nabla_y
\end{array}
\right):S^1{^{sc}TX}\rightarrow S^1_1{^{sc}TX},\quad\quad \mathfrak{D}_5=\left(\begin{array}{cc}
0 &\nabla_y^*
\end{array}
\right):S^1{^{sc}TX}\rightarrow C^\infty(X).
\]

%
\[
\sigma(\mathfrak{D}_4^*\mathfrak{D}_4-\mathfrak{D}_5^*\mathfrak{D}_5)=
\left(\begin{array}{cc}
0&0\\
0 &|\eta|^2-\eta\langle\eta,\,\cdot\,\rangle
\end{array}
\right)+R'(a),
\]
where $R'(a)$ is the principal symbol of an operator in $\mathrm{Diff}^1_{sc}(X,S^1{^{sc}TX};S^1{^{sc}TX})+\mathsf{F}\mathrm{Diff}^0_{sc}(X,S^1{^{sc}TX};S^1{^{sc}TX})$.
This proves
\begin{lemma}
\begin{equation}\label{laplaciandecomposition}
-\Delta^\mathcal{B}_\mathsf{F}=\frac{n-2}{n}(\nabla^*\nabla+\mathsf{F}^2)+\frac{1}{n}\mathfrak{D}_1^*\mathfrak{D}_1+\frac{1}{n}\mathfrak{D}_2^*\mathfrak{D}_2+\frac{1}{n}\mathfrak{D}_3\mathfrak{D}_3^*+\frac{2}{n}(\mathfrak{D}_4^*\mathfrak{D}_4-\mathfrak{D}_5^*\mathfrak{D}_5)+A+B.
\end{equation}
where $A\in\mathrm{Diff}^1_{sc}(X,S^1{^{sc}TX};S^1{^{sc}TX})$ and $B\in h^{-1}\mathrm{Diff}^0_{sc}(X,S^1{^{sc}TX};S^1{^{sc}TX})$. Here we have denoted $h=\frac{1}{\mathsf{F}}$.
\end{lemma}
For the subsequent analysis, we need to take various regions $\Omega_j$ with artificial boundary $\partial X$ and ``interior" boundary $\partial_{int}\Omega_j$. In particular, we will take $\Omega_0$ below such that $\partial_{int}\Omega_0\subset\partial M$.
Next, we prove the invertibility of $-\Delta^\mathcal{B}_\mathsf{F}$ on $\Omega_j$.

We will follows the lines in \cite{stefanov2018inverting}. Let $\dot{H}^{m,l}_{sc}(\Omega_j)$ be the subspace of $H^{m,l}_{sc}(X)$ consisting of distributions supported in $\overline{\Omega_j}$, and let $\bar{H}^{m,l}_{sc}(\Omega_j)$ be the space of restrictions of elements in $H^{m,l}_{sc}(X)$ to $\Omega_j$. The $\dot{H}^{m,l}_{sc}(\Omega_j)^*=\bar{H}^{-m,-l}_{sc}(\Omega_j)$.
\begin{lemma}\label{lemma7}
The operator $-\Delta^\mathcal{B}_\mathsf{F}$, considered as a map $\dot{H}^{1,0}_{sc}\rightarrow (\dot{H}^{1,0}_{sc})^*=\overline{H}^{-1,0}_{sc}$, is invertible for $\mathsf{F}>0$ sufficiently large.
\end{lemma}
\begin{proof}
Assume that $u=u_x\partial_x+u_y\partial_y\in S^1{^{sc}TX}$.
Using the identity \eqref{laplaciandecomposition},
we have
\[
\begin{split}
\|\mathrm{d}_\mathsf{F}^\mathcal{B}u\|_{L^2_{sc}}^2=\frac{n-2}{n}(\|\nabla u\|_{L^2_{sc}}^2+\mathsf{F}^2\|u\|_{L^2_{sc}}^2)&+\frac{1}{n}\|\mathfrak{D}_1u\|_{L^2_{sc}}^2+\frac{1}{n}\|\mathfrak{D}_2u\|_{L^2_{sc}}^2+\frac{1}{n}\|\mathfrak{D}_3^*u\|_{L^2_{sc}}^2\\
&+\frac{2}{n}(\|\nabla_yu_y\|_{L^2_{sc}}^2-\|\nabla_y^*u_y\|_{L^2_{sc}}^2)+\langle Au, u\rangle_{L^2_{sc}}+\langle R u, u\rangle_{L^2_{sc}}.
\end{split}
\]
We use the fact that
\[
\|\nabla_y^*u_y\|_{L^2_{sc}}^2\leq \|\nabla_yu_y\|_{L^2_{sc}}^2+\|u\|^2_{L^2_{sc}}.
\]
Since $A\in\mathrm{Diff}^1_{sc}(X)$ and $R\in h^{-1}\mathrm{Diff}^0_{sc}(X)$, we have
\[
\left|\langle Au, u\rangle\right|\leq C\|u\|_{\dot{H}^{1,0}_{sc}}\|u\|_{L_{sc}^2},\quad\left|\langle \tilde{R}u, u\rangle\right|\leq Ch^{-1}\|u\|_{L_{sc}^2}^2.
\]
Then we obtain
\[
\|\nabla u\|_{L^2_{sc}}^2+\mathsf{F}^2\|u\|_{L^2_{sc}}^2-\varepsilon\|u\|^2_{\dot{H}^{1,0}_{sc}}-\frac{C}{\varepsilon}\|u\|_{L^2_{sc}}^2-C\mathsf{F}\|u\|_{L^2_{sc}}^2\leq C\|\mathrm{d}_\mathsf{F}^\mathcal{B}u\|_{L^2_{sc}}^2,
\]
for some $\varepsilon>0$ small enough (one can just take $\varepsilon=\frac{1}{2}$). Then we can take $\mathsf{F}$ large enough to have
\[
\|\nabla u\|^2_{L^2_{sc}}+\|u\|^2_{L^2_{sc}}\leq C\|\mathrm{d}_\mathsf{F}^\mathcal{B}u\|^2_{L^2_{sc}}.
\]
Now, using the variational form
\[
\langle-\Delta_\mathsf{F}^\mathcal{B} u, v\rangle=\langle\mathrm{d}_\mathsf{F}^\mathcal{B}u,\mathrm{d}_\mathsf{F}^\mathcal{B}v \rangle,
\]
we have
\[
\|\mathrm{d}_\mathsf{F}^\mathcal{B}u\|^2_{L^2_{sc}}\leq \|\Delta_\mathsf{F}^\mathcal{B} u\|_{\bar{H}^{-1,0}_{sc}}\|u\|_{\dot{H}^{1,0}_{sc}}\leq C(\epsilon)\epsilon^{-1}\|\Delta_\mathsf{F}^\mathcal{B} u\|_{\bar{H}^{-1,0}_{sc}}^2+\epsilon \|u\|_{\dot{H}^{1,0}_{sc}}^2.
\]
for arbitrary $\epsilon>0$ and $C(\epsilon)>0$ depending on $\epsilon$. Now we have
\[
\|u\|_{\dot{H}^{1,0}_{sc}}=\|\nabla u\|^2_{L^2_{sc}}+\|u\|^2_{L^2_{sc}}\leq C(\epsilon)\epsilon^{-1}\|\Delta_\mathsf{F}^\mathcal{B} u\|_{\bar{H}^{-1,0}_{sc}}^2+\epsilon \|u\|_{\dot{H}^{1,0}_{sc}}^2.
\]
Taking $\epsilon$ small enough, we obtain
\[
\|u\|_{\dot{H}^{1,0}_{sc}}\leq C'\|\Delta_\mathsf{F}^\mathcal{B} u\|_{\bar{H}^{-1,0}_{sc}}
\]
with some constant $C'>0$.
This proves the invertibility.
\end{proof}

Define
\[
\begin{split}
\mathcal{S}_{\mathsf{F},\Omega_j}\phi&=\phi-\mathrm{d}_\mathsf{F}^\mathcal{B}(\Delta_{\mathsf{F},\Omega_j}^\mathcal{B})^{-1}\delta^\mathcal{B}_\mathsf{F}\phi,\\
\mathcal{P}_{\mathsf{F},\Omega_j}\phi&=\mathrm{d}_\mathsf{F}^\mathcal{B}(\Delta_{\mathsf{F},\Omega_j}^\mathcal{B})^{-1}\delta^\mathcal{B}_\mathsf{F}\phi,
\end{split}
\]
which are the solenoidal and potential projections of $\phi$ on $\Omega_j$. 

Now let $\Omega_2$ be a larger neighborhood of $\Omega$, and take $M\in\Psi_{sc}^{-3,0}(X;S^1{^{sc}TX},S^1{^{sc}TX})$ such that
\[
A_\mathsf{F}=N_\mathsf{F}+\mathrm{d}_\mathsf{F}^\mathcal{B}M\delta_\mathsf{F}^\mathcal{B}
\]
is elliptic in $\Psi^{-1,0}_{sc}(X;\mathcal{B}S^1_1{^{sc}TX},\mathcal{B}S^1_1{^{sc}TX})$ in $\Omega_2$.
Denote $\tilde{\Omega}_j=\overline{\Omega_j}\setminus\partial_{int}\Omega_j$. Take $\Omega\subset\Omega_1\subset\Omega_2$.
Take $G$ to be a parametrix of $A_\mathsf{F}$, where
\[
GA_\mathsf{F}=I+E,
\]
with $\mathrm{WF}_{sc}'(E)$ is disjoint from $\Omega_1$, and $E=-\mathrm{Id}$ near $\partial_{int}\Omega_2$. Now
\[
G(N_\mathsf{F}+\mathrm{d}_\mathsf{F}^\mathcal{B}M\delta_\mathsf{F}^\mathcal{B})=I+E
\]
Applying $\mathcal{S}_{\mathsf{F},\Omega_2}$ from both sides, we obtain
\[
\mathcal{S}_{\mathsf{F},\Omega_2}G(N_\mathsf{F}+\mathrm{d}_\mathsf{F}^\mathcal{B}M\delta_\mathsf{F}^\mathcal{B})\mathcal{S}_{\mathsf{F},\Omega_2}=\mathcal{S}_{\mathsf{F},\Omega_2}+\mathcal{S}_{\mathsf{F},\Omega_2}E\mathcal{S}_{\mathsf{F},\Omega_2}.
\]
Notice that $N_\mathsf{F}\mathcal{S}_{\mathsf{F},\Omega_2}=N_\mathsf{F}$ and $\delta_\mathsf{F}^\mathcal{B}\mathcal{S}_{\mathsf{F},\Omega_2}=0$, we end up with
\[
\mathcal{S}_{\mathsf{F},\Omega_2}GN_\mathsf{F}=\mathcal{S}_{\mathsf{F},\Omega_2}+\mathcal{S}_{\mathsf{F},\Omega_2}E\mathcal{S}_{\mathsf{F},\Omega_2}.
\]

Let $e_{12}$ be the (zero) extension map from $\Omega_1$ to $\Omega_2$ and $r_{21}$ be the restriction map from $\Omega_2$ to $\Omega_1$. Then
\[
r_{21}\mathcal{S}_{\mathsf{F},\Omega_2}GN_\mathsf{F}=r_{21}\mathcal{S}_{\mathsf{F},\Omega_2}GN_\mathsf{F}e_{12}=r_{21}\mathcal{S}_{\mathsf{F},\Omega_2}e_{12}+K_1,
\]
with 
\[
K_1=r_{21}\mathcal{S}_{\mathsf{F},\Omega_2}E\mathcal{S}_{\mathsf{F},\Omega_2}e_{12}.
\]

Now
\begin{equation}\label{SF1}
\begin{split}
S_{\mathsf{F},\Omega_1}-r_{21}\mathcal{S}_{\mathsf{F},\Omega_2}e_{12}=&-\mathrm{d}_\mathsf{F}^\mathcal{B}(\Delta_{\mathsf{F},\Omega_1}^\mathcal{B})^{-1}\delta^\mathcal{B}_\mathsf{F}+r_{21}\mathrm{d}_\mathsf{F}^\mathcal{B}(\Delta_{\mathsf{F},\Omega_2}^\mathcal{B})^{-1}\delta^\mathcal{B}_\mathsf{F}e_{12}\\
=&-\mathrm{d}_\mathsf{F}^\mathcal{B}(\Delta_{\mathsf{F},\Omega_1}^\mathcal{B})^{-1}\delta^\mathcal{B}_\mathsf{F}+\mathrm{d}_\mathsf{F}^\mathcal{B}r_{21}(\Delta_{\mathsf{F},\Omega_2}^\mathcal{B})^{-1}\delta^\mathcal{B}_\mathsf{F}e_{12}\\
=&-\mathrm{d}_\mathsf{F}^\mathcal{B}((\Delta_{\mathsf{F},\Omega_1}^\mathcal{B})^{-1}\delta^\mathcal{B}_\mathsf{F}-r_{21}(\Delta_{\mathsf{F},\Omega_2}^\mathcal{B})^{-1}\delta^\mathcal{B}_\mathsf{F}e_{12}).
\end{split}
\end{equation}
So
\[
r_{21}\mathcal{S}_{\mathsf{F},\Omega_2}GN_\mathsf{F}=S_{\mathsf{F},\Omega_1}+\mathrm{d}_\mathsf{F}^\mathcal{B}\left((\Delta_{\mathsf{F},\Omega_1}^\mathcal{B})^{-1}\delta^\mathcal{B}_\mathsf{F}-r_{21}(\Delta_{\mathsf{F},\Omega_2}^\mathcal{B})^{-1}\delta^\mathcal{B}_\mathsf{F}e_{12}\right)+K_1.
\]

Denote $\gamma_{\partial_{int}\Omega_1}$ to be the restriction operator to $\partial_{int}\Omega_1$. Similar to \cite[Lemma 4.9]{stefanov2018inverting}, we have
\begin{lemma}
Let $\psi\in H^{1/2,k}(\partial_{int}\Omega_j)$. For sufficiently large $\mathsf{F}>0$, there exists a unique solution $u\in\bar{H}^{1,k}_{sc}(\Omega_j)$ such that $\Delta_\mathsf{F}^\mathcal{B}u=0$, $\gamma_{\partial_{int}\Omega_j}u=\psi$. This defines the Poisson operator $B_{\Omega_j}:H^{1/2,k}_{sc}(\partial_{int}\Omega)\rightarrow\bar{H}^{1,k}_{sc}(\Omega_j)$ such that
\[
\Delta_\mathsf{F}^\mathcal{B}B_{\Omega_j}=0,\quad \gamma_{\partial_{int}\Omega_1}B_{\Omega_j}=\mathrm{Id},
\]
and for $s>1/2$ and $\phi\in C^\infty(\overline{\Omega_j})$ supported away from $\partial_{int}\Omega_j$, $\phi B_{\Omega_j}:H_{sc}^{s-1/2,r}(\partial_{int}\Omega_j)\rightarrow H^{s,r}_{sc}(\Omega_j)$.
\end{lemma}

Then we can write
\[
\begin{split}
r_{21}\mathcal{S}_{\mathsf{F},\Omega_2}GN_\mathsf{F}=\mathcal{S}_{\mathsf{F},\Omega_1}+\mathrm{d}_\mathsf{F}^\mathcal{B}\left((\Delta_{\mathsf{F},\Omega_1}^\mathcal{B})^{-1}\delta^\mathcal{B}_\mathsf{F}-r_{21}(\Delta_{\mathsf{F},\Omega_2}^\mathcal{B})^{-1}\delta^\mathcal{B}_\mathsf{F}e_{12}+B_{\Omega_1}\gamma_{\partial_{int}\Omega_1}(\Delta_{\mathsf{F},\Omega_2}^\mathcal{B})^{-1}\delta^\mathcal{B}_\mathsf{F}e_{12}\right)\\
-\mathrm{d}_\mathsf{F}^\mathcal{B}B_{\Omega_1}\gamma_{\partial_{int}\Omega_1}(\Delta_{\mathsf{F},\Omega_2}^\mathcal{B})^{-1}\delta^\mathcal{B}_\mathsf{F}e_{12}+K_1.
\end{split}
\]
Notice that
\[
\Delta_{\mathsf{F},\Omega_1}^\mathcal{B}\left[(\Delta_{\mathsf{F},\Omega_1}^\mathcal{B})^{-1}\delta^\mathcal{B}_\mathsf{F}-r_{21}(\Delta_{\mathsf{F},\Omega_2}^\mathcal{B})^{-1}\delta^\mathcal{B}_\mathsf{F}e_{12}+B_{\Omega_1}\gamma_{\partial_{int}\Omega_1}(\Delta_{\mathsf{F},\Omega_2}^\mathcal{B})^{-1}\delta^\mathcal{B}_\mathsf{F}e_{12}\right]=0\quad\text{in }\Omega_1,
\]
and
\[
\gamma_{\partial_{int}\Omega_1}\left((\Delta_{\mathsf{F},\Omega_1}^\mathcal{B})^{-1}\delta^\mathcal{B}_\mathsf{F}-r_{21}(\Delta_{\mathsf{F},\Omega_2}^\mathcal{B})^{-1}\delta^\mathcal{B}_\mathsf{F}e_{12}+B_{\Omega_1}\gamma_{\partial_{int}\Omega_1}(\Delta_{\mathsf{F},\Omega_2}^\mathcal{B})^{-1}\delta^\mathcal{B}_\mathsf{F}e_{12}\right)=0.
\]
Therefore
\[
\mathcal{S}_{\mathsf{F},\Omega_1}r_{21}\mathcal{S}_{\mathsf{F},\Omega_2}GN_\mathsf{F}=\mathcal{S}_{\mathsf{F},\Omega_1}-\mathcal{S}_{\mathsf{F},\Omega_1}\mathrm{d}_\mathsf{F}^\mathcal{B}B_{\Omega_1}\gamma_{\partial_{int}\Omega_1}(\Delta_{\mathsf{F},\Omega_2}^\mathcal{B})^{-1}\delta^\mathcal{B}_\mathsf{F}e_{12}+\mathcal{S}_{\mathsf{F},\Omega_1}K_1.
\]
Now write $\Omega=\Omega_0$, and $e_{0j}$ for the extension map from $\Omega_0$ to $\Omega_j$. Composing from the right, we obtain
\[
\mathcal{S}_{\mathsf{F},\Omega_1}r_{21}\mathcal{S}_{\mathsf{F},\Omega_2}GN_\mathsf{F}=\mathcal{S}_{\mathsf{F},\Omega_1}-\mathcal{S}_{\mathsf{F},\Omega_1}\mathrm{d}_\mathsf{F}^\mathcal{B}B_{\Omega_1}\gamma_{\partial_{int}\Omega_1}(\Delta_{\mathsf{F},\Omega_2}^\mathcal{B})^{-1}\delta^\mathcal{B}_\mathsf{F}e_{02}+\mathcal{S}_{\mathsf{F},\Omega_1}K_1e_{01}.
\]

Using the arguments as in \cite[Lemma 4.10, 4.11]{stefanov2018inverting}, we know that 
\[
K_2=-\mathcal{S}_{\mathsf{F},\Omega_1}\mathrm{d}_\mathsf{F}^\mathcal{B}B_{\Omega_1}\gamma_{\partial_{int}\Omega_1}(\Delta_{\mathsf{F},\Omega_2}^\mathcal{B})^{-1}\delta^\mathcal{B}_\mathsf{F}e_{02}+\mathcal{S}_{\mathsf{F},\Omega_1}K_1e_{01}
\]
is smoothing and small if $\Omega\subset\{x\leq \delta\}$ with $\delta$ small enough. In particular, for any $\epsilon>0$, we can take $\delta$ small enough such that
\[
\|K_2\|_{\mathcal{L}(x^kL^2_{sc}(\Omega),\bar{H}_{sc}^{s,r}(\Omega_1))}\leq\epsilon.
\]
Now we have
\begin{equation}\label{SF0}
\mathcal{S}_{\mathsf{F},\Omega_1}r_{21}\mathcal{S}_{\mathsf{F},\Omega_2}GN_\mathsf{F}=\mathcal{S}_{\mathsf{F},\Omega_1}e_{01}+K_2.
\end{equation}

We need the following lemma:

\begin{lemma}\label{est_extension}
For any $\mathsf{F}>0$ and $r\in\mathbb{R}$, 
\[
\|v\|_{\bar{H}^{1,r}_{sc}(\Omega_j)}\leq C\left(\|x^{-r}\mathrm{d}^\mathcal{B}_\mathsf{F}v\|_{L^2_{sc}(\Omega_j)}+\|v\|_{x^{-r}L^2_{sc}(\Omega_j)}\right),
\]
for $(1,0)$-tensors $v\in\bar{H}^{1,r}_{sc}(\Omega_j)$.
\end{lemma}
\begin{proof}
Similar to the proof of \cite[Lemma 4.5]{stefanov2018inverting}, we only need to prove the case $r=0$.

The key is to prove the existence of an extension. 
 Let $\tilde{\Omega}_j$ be a domain in $X$ with $C^\infty$ boundary, transversal to $\partial X$, containing $\overline{\Omega_j}$. We claim that there exists a continuous map 
\[
E:\bar{H}^{1,0}_{sc}(\Omega_j)\rightarrow \dot{H}^{1,0}_{sc}(\tilde{\Omega}_j) 
\]
such that
\begin{equation}\label{extension_vector_est}
\|\mathrm{d}_\mathsf{F}^\mathcal{B}Ev\|_{L^2_{sc}(\tilde{\Omega}_j)}+\|Ev\|_{L^2_{sc}(\tilde{\Omega}_j)}\leq C\left(\|\mathrm{d}_\mathsf{F}^\mathcal{B}v\|_{L^2_{sc}(\Omega_j)}+\|v\|_{L^2_{sc}(\Omega_j)}\right),\quad v\in\bar{H}^{1,0}_{sc}(\Omega_j),
\end{equation}
where $v$ is a scattering $(1,0)$-tensor .

We can use the local diffeomorphism of $X$ to $\overline{\mathbb{R}^n}$, and $\overline{\Omega_j}$ to $\overline{\mathbb{R}^n_+}$ to use the extension map 
\[
E_1:C^1(\overline{\mathbb{R}^n_+})\rightarrow C^1(\overline{\mathbb{R}^n}),
\]
which will be defined below.

Take $\Phi_k(x',x_n)=(x',-kx_n)$ for $x_n<0$ and define
\[
\begin{split}
\left(E_1\sum v^i\partial_{x^i}\right)(x',x_n)&=\sum_{k=1}^3c_k\Phi^*_k\left(\sum v^i\partial_{x^i}\right),\quad x_n<0,\\
\left(E_1\sum v^i\partial_{x^i}\right)(x',x_n)&=\sum v^i\partial_{x^i},\quad x_n\geq 0.
\end{split}
\]

Then $\Phi^*_k$ acts on $1$-vector as
\[
\begin{split}
\Phi^*_kv^i\partial_{x^i}&=v^i(x',-kx_n)\partial_{x^i},\quad i\neq n,\\
\Phi^*_kv^n\partial_{x^n}&=-\frac{1}{k}v^n(x',-kx_n)\partial_{x^n},\\
\partial_j\Phi^*_kv^i\partial_{x^i}&=(\partial_jv^i)(x',-kx_n)\partial_{x^i},\quad i, j\neq n,\\
\partial_j\Phi^*_kv^n\partial_{x^n}&=-\frac{1}{k}(\partial_jv^n)(x',-kx_n)\partial_{x^n},\quad j\neq n,\\
\partial_n\Phi^*_kv^i\partial_{x^i}&=-k(\partial_nv^i)(x',-kx_n)\partial_{x^i},\quad i\neq n,\\
\partial_n\Phi^*_kv^n\partial_{x^n}&=(\partial_nv^n)(x',-kx_n)\partial_{x^n}.
\end{split}
\]
So to match the derivatives at $x_n=0$, we need to take $c_1,c_2,c_3$ to satisfy the equation
\[
\left(\begin{array}{ccc}
-1 & -\frac{1}{2} & -\frac{1}{3}\\
1 & 1 &1\\
-1 & -2 &-3
\end{array}\right)\left(\begin{array}{c}
c_1\\
c_2\\
c_3
\end{array}
\right)=\left(\begin{array}{c}
1\\
1\\
1
\end{array}\right).
\]
The matrix on the left hand side is Vandermode and is thus invertible.
With $c_k$ chosen as above, we have that $E_1:C_c^\infty(\mathbb{R}_+^n)\rightarrow C_c^1(\mathbb{R}^n)$ can be extended continuously to a map $H^1(\mathbb{R}^n_+)\rightarrow H^1(\mathbb{R}^n)$.

Notice that $\Phi^*_k$ acts on $1+1$ tensors in the following way
\[
\begin{split}
\Phi^*_ku^i_j\partial_{x^i}\otimes\mathrm{d}x^j&=u^i_j(x',-kx_n)\partial_{x^i}\otimes\mathrm{d}x^j,\quad i,j\neq n,\\
\Phi^*_ku^i_n\partial_{x^i}\otimes\mathrm{d}x^n&=-ku^i_n(x',-kx_n)\partial_{x^i}\otimes\mathrm{d}x^n,\quad i\neq n,\\
\Phi^*_ku^n_j\partial_{x^n}\otimes\mathrm{d}x^j&=-\frac{1}{k}u^n_j(x',-kx_n)\partial_{x^n}\otimes\mathrm{d}x^j,\quad j\neq n,\\
\Phi^*_ku^n_n\partial_{x^n}\otimes\mathrm{d}x^n&=u^n_n(x',-kx_n)\partial_{x^n}\otimes\mathrm{d}x^n.
\end{split}
\]
We can then verify that
\[
\begin{split}
(\partial_j\Phi^*_kv^i\partial_{x^i})\otimes\mathrm{d}x^j&=\partial_ju(x',-kx_n)\partial_{x^i}\otimes\mathrm{d}x^j=\Phi^*_k(\partial_jv^i\partial_{x^i}\otimes\mathrm{d}x^j),\quad i,j\neq n,\\
(\partial_n\Phi^*_kv^i\partial_{x^i})\otimes\mathrm{d}x^n&=-k\partial_nu(x',-kx_n)\partial_{x^i}\otimes\mathrm{d}x^n=\Phi^*_k(\partial_nv^i\partial_{x^j}\otimes\mathrm{d}x^n),\quad i\neq n,\\
(\partial_j\Phi^*_kv^n\partial_{x^n})\otimes\mathrm{d}x^j&=-\frac{1}{k}\partial_ju(x',-kx_n)\partial_{x^n}\otimes\mathrm{d}x^j=\Phi^*_k(\partial_jv^n\partial_{x^n}\otimes\mathrm{d}x^j),\quad j\neq n,\\
(\partial_n\Phi^*_kv^n\partial_{x^n})\otimes\mathrm{d}x^n&=\partial_nu(x',-kx_n)\partial_{x^n}\otimes\mathrm{d}x^n=\Phi^*_k(\partial_nv^n\partial_{x^n}\otimes\mathrm{d}x^n).
\end{split}
\]
Therefore, with $g$ replaced by a translation-invariant Riemannian metric $g_0$ (e.g., the Euclidean one), we have
\[
\mathrm{d}\Phi^*_k=\Phi^*_k\mathrm{d}.
\]
 Note that then
\[
\mathrm{d}_{g_0}^\mathcal{B}v=\mathrm{d}v-\frac{1}{n}\mathrm{trace}(\mathrm{d}_{g_0}v)\delta^i_j.
\]
This yields
\[
\|\mathrm{d}_{g_0}^\mathcal{B}\Phi^*_kv\|_{L^2(\mathbb{R}^n_-)}\leq C\|\mathrm{d}_{g_0}^\mathcal{B}v\|_{L^2(\mathbb{R}^n_+)},
\]
and consequently
\[
\|\mathrm{d}_{g_0}^\mathcal{B}E_1v\|_{L^2(\mathbb{R}^n)}\leq C\|\mathrm{d}_{g_0}^\mathcal{B}v\|_{L^2(\mathbb{R}^n_+)}.
\]

Then one use a partition of unity $\{\rho_k\}$ to localize on $\Omega_j$ to finish the proof of \eqref{extension_vector_est} as in the proof of \cite[Lemma 4.5]{stefanov2018inverting}.  Locally identifying $\mathrm{supp}\rho_k$ with $\overline{\mathbb{R}^n_+}$ and $X$ with $\overline{\mathbb{R}^n}$, we denote $E_{1,k}$ the extension map from $\overline{\mathbb{R}^n_+}$ to $X$.  Take $\psi_k$ to be identically $1$ near the support of $\rho_k$. Then
\[
\sum\psi_kE_{1,k}\rho_k
\]
is an extension map from $H^1(\Omega_j)$ to $H^1(\tilde{\Omega}_j)$.
Note that while $\mathrm{d}=\mathrm{d}_g$ depends on the metric $g$ we have locally,
\[
\|\mathrm{d}^\mathcal{B}\psi_kE_{k,1}\rho_kv\|_{L^2(\mathbb{R}^n)}+\|E_{k,1}\rho_kv\|_{L^2(\mathbb{R}^n)}\leq C\|\mathrm{d}^\mathcal{B}\rho_kv\|_{L^2(\mathbb{R}^n_+)}+\|\rho_kv\|_{L^2(\mathbb{R}^n_+)},
\]
since $\mathrm{d}^\mathcal{B}_gv=\mathrm{d}^\mathcal{B}_{g_0}v+Rv$ with a zeroth order $R$. Since $\mathrm{d}^\mathcal{B}_\mathsf{F}$ differs from $\mathrm{d}^\mathcal{B}$ by a zeroth order term, we have
\[
\|\mathrm{d}^\mathcal{B}_\mathsf{F}\psi_kE_{k,1}\rho_kv\|_{L^2(\mathbb{R}^n)}+\|E_{k,1}\rho_kv\|_{L^2(\mathbb{R}^n)}\leq C\|\mathrm{d}_\mathsf{F}^\mathcal{B}\rho_kv\|_{L^2(\mathbb{R}^n_+)}+\|\rho_kv\|_{L^2(\mathbb{R}^n_+)}.
\]
Summing over $k$ yields \eqref{extension_vector_est}. With this extension found, we can prove
\[
\|v\|_{\bar{H}^{1,0}_{sc}(\Omega_j)}\leq C\left(\|\mathrm{d}^\mathcal{B}_\mathsf{F}v\|_{L^2_{sc}(\Omega_j)}+\|v\|_{L^2_{sc}(\Omega_j)}\right),
\]
and completes the proof of the lemma.
\end{proof}

Next we prove a mapping property for $\mathrm{d}_\mathsf{F}^\mathcal{B}$.
\begin{lemma}
The map
\[
\mathrm{d}_\mathsf{F}^\mathcal{B}:\dot{H}_{sc}^{1,r}(\Omega_1\setminus\Omega)\rightarrow H^{0,r}_{sc}(\Omega_1\setminus\Omega)
\]
is injective, with a continuous left inverse $P_{\Omega_1\setminus\Omega}:H^{0,r}_{sc}(\Omega_1\setminus\Omega)\rightarrow H^{1,r-2}_{sc}(\Omega_1\setminus\Omega)$, where $r\leq-\frac{n-3}{2}$.
\end{lemma}

\begin{proof}
We will use the identity
\begin{equation}\label{id_integralmixed}
v(\gamma(s))^i\eta_i(s)=\int_0^s[\mathrm{d}^\mathcal{B}v(\gamma(t))]^i_j\eta_i(t)\dot{\gamma}^j(t)\mathrm{d}t,
\end{equation}
where $\gamma$ is a unit speed geodesic on $(\Omega_1,g)$ with $\gamma(0)\in\partial_{int}\Omega_1$ and $\gamma(\tau)\in\partial_{int}\Omega\cup\partial X$, with $\gamma([0,\tau])\subset\Omega_1\setminus\Omega$, $\eta$ is a covector field on $\gamma$ conormal to $\dot{\gamma}$. The  identity \eqref{id_integralmixed} is a result of the Fundamental Theorem of Calculus.

As in the proof of \cite[Lemma 4.13]{stefanov2018inverting}, we choose $\gamma$ such that $-\frac{\partial}{\partial t}(x\circ\gamma)$ is bounded below and above by a positive constant. Then $(x\circ\gamma)^2\frac{\partial}{\partial t}(x^{-1}\circ\gamma)$ is also bounded below and above. We can take a smooth family of such geodesics emanating from $\partial_{int}\Omega_1$, parametrized by $\partial_{int}\Omega_1$. Take a measure $\mathrm{d}\omega$ on $\partial_{int}\Omega_1$ such that $\mathrm{d}\omega\mathrm{d}t$ is equivalent to the volume form $\mathrm{d}g$.

For any $k\geq 0$ and $t\in[0,s]$, we have
\[
\begin{split}
&|e^{-\mathsf{F}/(x(\gamma(s)))}x(\gamma(s))^k[v(\gamma(s))]^i\eta_i(s)|^2\\
=&\left|\int_0^se^{-\mathsf{F}/(x(\gamma(s)))}x(\gamma(t))^{k+1}[\mathrm{d}^\mathcal{B}v(\gamma(t))]^i_{j}\eta_i(t)\dot{\gamma}^j(t)e^{-\mathsf{F}(1/x(\gamma(s))-1/x(\gamma(t)))}x(\gamma(t))^{-1}\mathrm{d}t\right|^2\\
\leq &\left(\int_0^\tau e^{-2\mathsf{F}/(x(\gamma(t))}x(\gamma(t))^{2k+2}\left|[\mathrm{d}^\mathcal{B}v(\gamma(t))]^i_j\eta_i(t)\dot{\gamma}^j(t)\right|^2\mathrm{d}t\right)\\
&\times\left(\int_0^se^{-2\mathsf{F}(1/x(\gamma(s))-1/x(\gamma(t)))}x(\gamma(t))^{-2}\mathrm{d}t\right).
\end{split}
\]
Assume $\eta=\eta_x\mathrm{d}x+\eta_y\cdot\mathrm{d}y$ with $|\eta_x|+|\eta_y|<C$, so
\[
\left|[\mathrm{d}^\mathcal{B}v(\gamma(t))]^i_j\eta_i(t)\dot{\gamma}^j(t)\right|\leq Cx(\gamma(t))^{-1}\left|\mathrm{d}^\mathcal{B}v(\gamma(t))\right|_{g_{sc}}.
\]
Because that $(x\circ\gamma)^2\frac{\partial}{\partial t}(x^{-1}\circ\gamma)$ is bounded below by a positive constant, we have
\[
\begin{split}
\int_0^se^{-2\mathsf{F}(1/x(\gamma(s))-1/x(\gamma(t)))}x(\gamma(t))^{-2}\mathrm{d}t\leq & C\int_0^se^{-2\mathsf{F}(1/x(\gamma(s))-1/x(\gamma(t)))}\frac{\partial}{\partial t}(x^{-1}(\gamma(t)))\mathrm{d}t\\
\leq &C\int_{r_0}^{x^{-1}(\gamma(s))}e^{-2\mathsf{F}(1/x(\gamma(s))-r)}\mathrm{d}r\\
\leq &C\mathsf{F}^{-1}.
\end{split}
\]
Therefore
\[
\left|e^{-\mathsf{F}/(x(\gamma(s)))}x(\gamma(s))^k[v(\gamma(s))]^i\eta_i(s)\right|^2\leq C\mathsf{F}^{-1}\int_0^\tau e^{-2\mathsf{F}/(x(\gamma(t))}x(\gamma(t))^{2k}\left|\mathrm{d}^\mathcal{B}v(\gamma(t))\right|_{g_{sc}}^2\mathrm{d}t.
\]
Integrating above inequality over the family of geodesics, we have
\[
\|e^{-\mathsf{F}/x}x^k\mathsf{X}^\perp(v)\|^2_{L^2(\Omega_1\setminus\Omega)}\leq C\mathsf{F}^{-1}\|x^{k}e^{-\mathsf{F}/x}\mathrm{d}^\mathcal{B}v\|^2_{L^2(\Omega_1\setminus\Omega;\,\mathcal{B}_1^1{^{sc}TX})},
\]
where $\mathsf{X}^\perp=\eta_x\mathrm{d}x+\eta_y\cdot\mathrm{d}y$ is chosen to be a smooth section of $T^*(\Omega_1\setminus\overline{\Omega})$ conormal to the tangent vector of the family of geodesics $\mathsf{X}$. Using different families of geodesics and choosing different conormal vectors, such that varying $(\eta_x,\eta_y)$ span $\mathbb{R}^n$, we have
\[
\|x^ke^{-\mathsf{F}/x}v\|^2_{L^2(\Omega_1\setminus\Omega;\,\mathcal{B}^1{^{sc}TX})}\leq C\mathsf{F}^{-1}\|x^{k}e^{-\mathsf{F}/x}\mathrm{d}^\mathcal{B}v\|^2_{L^2(\Omega_1\setminus\Omega;\,\mathcal{B}_1^1{^{sc}TX})}.
\]
Using the estimate
\[
|v(p)|_{g_{sc}}\leq Cx^{-2}(p)|v(p)|_{g},
\]
one can change the volume form to obtain
\[
\|x^{k+(n+1)/2}e^{-\mathsf{F}/x}v\|^2_{L_{sc}^2(\Omega_1\setminus\Omega;\,\mathcal{B}^1{^{sc}TX})}\leq C\mathsf{F}^{-1}\|x^{k-2+(n+1)/2}e^{-\mathsf{F}/x}\mathrm{d}^\mathcal{B}v\|^2_{L_{sc}^2(\Omega_1\setminus\Omega;\,\mathcal{B}_1^1{^{sc}TX})}.
\]

With $u=e^{-\mathsf{F}/x}v$, we have the following Poincar\'e estimate
\[
\|u\|^2_{H_{sc}^{0,r-2}(\Omega_1\setminus\Omega)}\leq C\mathsf{F}^{-1}\|\mathrm{d}_\mathsf{F}^\mathcal{B}u\|^2_{H^{0,r}_{sc}(\Omega_1\setminus\Omega)},
\]
with $r\leq -(n-3)/2$, for $u\in\dot{H}^{1,r}_{sc}(\Omega_1\setminus\Omega)$.

Together with the estimates in Lemma \ref{est_extension},
\[
\|u\|_{\bar{H}^{1,r-2}_{sc}(\Omega_1\setminus\Omega)}\leq C\left(\|\mathrm{d}^\mathcal{B}_\mathsf{F}u\|_{H^{0,r}_{sc}(\Omega_1\setminus\Omega)}+\|u\|_{H^{0,r}_{sc}(\Omega_1\setminus\Omega)}\right),
\]
we end up with
\[
\|u\|_{\dot{H}^{1,r-2}_{sc}(\Omega_1\setminus\Omega)}\leq C\|\mathrm{d}_\mathsf{F}^\mathcal{B}u\|^2_{H^{0,r}_{sc}(\Omega_1\setminus\Omega)},\quad\quad u\in \dot{H}^{1,r}_{sc}(\Omega_1\setminus\Omega).
\]
\end{proof}

In the following, we need to work in a sufficiently negatively weighted space $H^{0,r}_{sc}$, with $r\leq -(n-3)/2$. We start to compute, as in \eqref{SF1},
\[
\mathcal{S}_{\mathsf{F},\Omega}-r_{10}\mathcal{S}_{\mathsf{F},\Omega_1}e_{01}=-\mathrm{d}_\mathsf{F}^\mathcal{B}((\Delta_{\mathsf{F},\Omega}^\mathcal{B})^{-1}\delta^\mathcal{B}_\mathsf{F}-r_{10}(\Delta_{\mathsf{F},\Omega_1}^\mathcal{B})^{-1}\delta^\mathcal{B}_\mathsf{F}e_{01}).
\]

Notice that $u=((\Delta_{\mathsf{F},\Omega}^\mathcal{B})^{-1}\delta^\mathcal{B}_\mathsf{F}-r_{10}(\Delta_{\mathsf{F},\Omega_1}^\mathcal{B})^{-1}\delta^\mathcal{B}_\mathsf{F}e_{01})f$ is the solution to the Dirichlet problem
\[
\Delta_{\mathsf{F}}^\mathcal{B}u=0,\quad u\vert_{\partial\Omega}=-\gamma_{\partial_{int}\Omega}(\Delta_{\mathsf{F},\Omega_1}^\mathcal{B})^{-1}\delta^\mathcal{B}_\mathsf{F}e_{01}f,
\]
so
\[
u=-B_{\Omega}\gamma_{\partial_{int}\Omega}(\Delta_{\mathsf{F},\Omega_1}^\mathcal{B})^{-1}\delta^\mathcal{B}_\mathsf{F}e_{01}f.
\]
Therefore
\[
\begin{split}
\mathcal{S}_{\mathsf{F},\Omega}-r_{10}\mathcal{S}_{\mathsf{F},\Omega_1}e_{01}=&-\mathrm{d}_\mathsf{F}^\mathcal{B}B_{\Omega}\gamma_{\partial_{int}\Omega}(\Delta_{\mathsf{F},\Omega_1}^\mathcal{B})^{-1}\delta^\mathcal{B}_\mathsf{F}e_{01}\\
=&-\mathrm{d}_\mathsf{F}^\mathcal{B}B_{\Omega}\gamma_{\partial_{int}\Omega}P_{\Omega_1\setminus\Omega}\mathrm{d}_\mathsf{F}^\mathcal{B}(\Delta_{\mathsf{F},\Omega_1}^\mathcal{B})^{-1}\delta^\mathcal{B}_\mathsf{F}e_{01}\\
=&-\mathrm{d}_\mathsf{F}^\mathcal{B}B_{\Omega}\gamma_{\partial_{int}\Omega}P_{\Omega_1\setminus\Omega}\mathcal{S}_{\mathsf{F},\Omega_1}e_{01}.
\end{split}
\]
Together with \eqref{SF0}, we have
\[
r_{10}\mathcal{S}_{\mathsf{F},\Omega_1}r_{21}\mathcal{S}_{\mathsf{F},\Omega_2}GN_\mathsf{F}=\mathcal{S}_{\mathsf{F},\Omega}+\mathrm{d}_\mathsf{F}^\mathcal{B}B_{\Omega}\gamma_{\partial_{int}\Omega}P_{\Omega_1\setminus\Omega}\mathcal{S}_{\mathsf{F},\Omega_1}e_{01}+r_{10}K_2.
\]
Using \eqref{SF0} again for the term $\mathcal{S}_{\mathsf{F},\Omega_1}e_{01}$ on the right, we get
\[
r_{10}\mathcal{S}_{\mathsf{F},\Omega_1}r_{21}\mathcal{S}_{\mathsf{F},\Omega_2}GN_\mathsf{F}=\mathcal{S}_{\mathsf{F},\Omega}+\mathrm{d}_\mathsf{F}^\mathcal{B}B_{\Omega}\gamma_{\partial_{int}\Omega}P_{\Omega_1\setminus\Omega}(\mathcal{S}_{\mathsf{F},\Omega_1}r_{21}\mathcal{S}_{\mathsf{F},\Omega_2}GN_\mathsf{F}-K_2)+r_{10}K_2.
\]
Now we can write
\[
(r_{10}-\mathrm{d}_\mathsf{F}^\mathcal{B}B_{\Omega}\gamma_{\partial_{int}\Omega}P_{\Omega_1\setminus\Omega})\mathcal{S}_{\mathsf{F},\Omega_1}r_{21}\mathcal{S}_{\mathsf{F},\Omega_2}GN_\mathsf{F}=\mathcal{S}_{\mathsf{F},\Omega}+(r_{10}-\mathrm{d}_\mathsf{F}^\mathcal{B}B_{\Omega}\gamma_{\partial_{int}\Omega}P_{\Omega_1\setminus\Omega})K_2.
\]
Adding $\mathcal{P}_{\mathsf{F},\Omega}$ to both sides, we get
\[
(r_{10}-\mathrm{d}_\mathsf{F}^\mathcal{B}B_{\Omega}\gamma_{\partial_{int}\Omega}P_{\Omega_1\setminus\Omega})\mathcal{S}_{\mathsf{F},\Omega_1}r_{21}\mathcal{S}_{\mathsf{F},\Omega_2}GN_\mathsf{F}+\mathcal{P}_{\mathsf{F},\Omega}=\mathrm{Id}+(r_{10}-\mathrm{d}_\mathsf{F}^\mathcal{B}B_{\Omega}\gamma_{\partial_{int}\Omega}P_{\Omega_1\setminus\Omega})K_2.
\]
Using the smallness of $K_2$, we have the invertibility of $\mathrm{Id}+(r_{10}-\mathrm{d}_\mathsf{F}^\mathcal{B}B_{\Omega}\gamma_{\partial_{int}\Omega}P_{\Omega_1\setminus\Omega})K_2$, therefore
\[
(\mathrm{Id}+(r_{10}-\mathrm{d}_\mathsf{F}^\mathcal{B}B_{\Omega}\gamma_{\partial_{int}\Omega}P_{\Omega_1\setminus\Omega})K_2)^{-1}\left((r_{10}-\mathrm{d}_\mathsf{F}^\mathcal{B}B_{\Omega}\gamma_{\partial_{int}\Omega}P_{\Omega_1\setminus\Omega})\mathcal{S}_{\mathsf{F},\Omega_1}r_{21}\mathcal{S}_{\mathsf{F},\Omega_2}GN_\mathsf{F}+\mathcal{P}_{\mathsf{F},\Omega}\right)=\mathrm{Id}.
\]
Applying $\mathcal{S}_{\mathsf{F},\Omega}$ from the right yields
\begin{equation}\label{finalreconstructionformula}
\mathcal{S}_{\mathsf{F},\Omega}=(\mathrm{Id}+(r_{10}-\mathrm{d}_\mathsf{F}^\mathcal{B}B_{\Omega}\gamma_{\partial_{int}\Omega}P_{\Omega_1\setminus\Omega})K_2)^{-1}(r_{10}-\mathrm{d}_\mathsf{F}^\mathcal{B}B_{\Omega}\gamma_{\partial_{int}\Omega}P_{\Omega_1\setminus\Omega})\mathcal{S}_{\mathsf{F},\Omega_1}r_{21}\mathcal{S}_{\mathsf{F},\Omega_2}GN_\mathsf{F},
\end{equation}
where we have used the fact that $N_\mathsf{F}=N_\mathsf{F}\mathcal{S}_{\mathsf{F},\Omega}$ and $\mathcal{P}_{\mathsf{F},\Omega}\mathcal{S}_{\mathsf{F},\Omega}=0$.

Now assume $f_\mathsf{F}=e^{-\mathsf{F}/x}f$ satisfies the gauge condition $\delta_\mathsf{F}f_\mathsf{F}=0$. If $L_{1,1}f=0$, then $N_\mathsf{F}f_\mathsf{F}=0$, we have
\[
f_\mathsf{F}=\mathcal{S}_{\mathsf{F},\Omega}f_\mathsf{F}=0
\]
by applying \eqref{finalreconstructionformula} to $f_\mathsf{F}$.
This shows the $s$-injectivity of $L_{1,1}$ within the $e^{-2\mathsf{F}/x}$-solenoidal gauge. This means that 
\[
f=\mathrm{d}^\mathcal{B}v,
\]
for a $1$-tensor $v$ with $v\vert_{\partial M\cap X}=0$.

This gives
\begin{theorem}\label{localuniqueness11}
Assume $\partial M$ is strictly convex at $p\in\partial M$. There exists a function $\tilde{x}\in C^\infty(\widetilde{M})$ with $O_p=\{\tilde{x}>-c\}\cap M$ for sufficiently small $c>0$, such that for a given trace-free $(1,1)$-tensor $f$, there exists a $1$-tensor $v$ with $v\vert_{O_p\cap\partial M}=0$ such that $f-\mathrm{d}^\mathcal{B}v$ can be uniquely determined by $L_{1,1}f$ restricted to $O_p$-local geodesics.
\end{theorem}
Using the above local result, the global $s$-injectivity can be proved as in \cite{paternain2019geodesic}.
\begin{theorem}
Assume $\partial M$ is strictly convex and $(M,g)$ admits a smooth strictly convex function $\rho$. If
\[
L_{1,1}f(\gamma,\eta)=0
\]
for any geodesic $\gamma$ in $M$ with endpoints on $\partial M$, and $\eta$ a parallel vector field conormal to $\gamma$. Then 
\[
f=\mathrm{d}^\mathcal{B}v,\quad v\vert_{\partial M}=0.
\]
\end{theorem}

\begin{proof}[Sketch of proof]
Assume $b=\sup_Mf$.
By the local uniqueness result Theorem \ref{localuniqueness11}, for any $p\in \rho^{-1}(b)\subset M$, there exists a wedge-shaped neighborhood $O_p$ in $M$ and a vector field $v_p$ satisfying
\[
f=\mathrm{d}^\mathcal{B}v_p,\quad v_p\vert_{O_p\cap\partial M}=0.
\]
Choose an arbitrary point $z\in O_p$, one can choose a geodesic $\gamma_{z,v}$ connecting $z$ to $O_p\cap \partial M$, where $\gamma_{z,v}(t)=z,\dot{\gamma}_{z,v}=v$. Choose $\eta(t)$ to be a parallel covector field on $\gamma_z$ such that $\eta(0)=\eta$ and $\eta(t)$ is conormal to $\dot{\gamma}(t)$. Then we have
\[
\langle f(\gamma(t)), \eta(t)\rangle=X\langle v_p(\gamma(t)),\eta(t)\rangle, \quad \langle v_p(\gamma(t)),\eta(t)\rangle\vert_{\tau(z,v)}=0,
\]
where $X$ is the vector field of the geodesic $\gamma_{z,v}$ and $\tau(z,v)$ is the exit time of $\gamma_{z,v}$.

By compactness of $\rho^{-1}(b)$, we can choose a finite cover $O_{p_j}$, $j=1,\cdots, m$, for some $p_1,\cdots,p_m\in\rho^{-1}(b)$ such that
\[
\rho^{-1}(b)\subset \cup_{j=1}^mO_{p_j},
\]
and
\[
\rho^{-1}{(b-\varepsilon,b)}\subset \cup_{j=1}^mO_{p_j}
\]
for some $\varepsilon>0$.
If $z\in O_{p_j}\cap O_{p_k}\neq \emptyset$, we have
\[
X\langle (v_{p_j}-v_{p_k})(\gamma(t)),\eta(t)\rangle=0\text{ in }O_{p_j}\cap O_{p_k},\quad \langle (v_{p_j}-v_{p_k})(\gamma(t)),\eta(t)\rangle\vert_{\tau(z,v)}=0.
\]
By the standard ODE theory, we have
\[
\langle(v_{p_j}-v_{p_k})(z),\eta\rangle=0.
\]
Then by perturbing $v$ and choosing different $\eta$'s, we can conclude that
\[
(v_{p_j}-v_{p_k})(z)=0.
\]
Therefore there exists $v\in \rho^{-1}{(b-\varepsilon,b)}$ such that
\[
f=\mathrm{d}^\mathcal{B}v,\quad v\vert_{\rho^{-1}{(b-\varepsilon,b)}\cap\partial M}=0.
\]

Using similar proof as for \cite[Proposition 6.2]{paternain2019geodesic}, one can show that
\[
f=\mathrm{d}^\mathcal{B}v\text{ in }M_{>a}
\]
for some $v\vert_{\partial M}=0$. Here $M_{>a}:=\rho^{-1}{(a,+\infty)}$ with $a=\inf_M f$. Dealing with the case $M_{>a}=M\setminus\{x_0\}$ as in the proof of \cite[Theorem 1.6]{paternain2019geodesic}, one can finish the proof of the theorem.
\end{proof}

\bibliographystyle{abbrv}
\bibliography{biblio}
\end{document}